\documentclass[]{article}
\usepackage{graphicx}
\usepackage{graphics}
\usepackage{amssymb,amsmath,amsthm,epsfig}
\newcommand{\h}{\mathcal{H}}
\newcommand{\wJ}{\widetilde P}

\newtheorem{theorem}{Theorem}
\newtheorem{lemma}{Lemma}
\newtheorem{corollary}{Corollary}
\newtheorem{definition}{Definition}
\newtheorem{proposition}{Proposition}

\newtheorem{remark}{Remark}
\newcommand{\norm}[1]{{\left\|{#1}\right\|}}
\newcommand{\R}{\mathbb R}

\newcommand{\ps}{\psi_{n,c}^{(\alpha)}}
\newcommand{\scal}[1]{{\left\langle{#1}\right\rangle}}
\newcounter{reh}
\newcounter{rek}
\vspace{5mm}
\begin{document}
		\begin{center}
		{\large {\bf Bochner-Riesz Means Convergence of Prolate Spheroidal Series and Their Extensions}}\\
		\vskip 1cm Mourad Boulsane$^a$  {\footnote{
				Corresponding author: Mourad Boulsane, Email: boulsane.mourad@hotmail.fr}}\& Ahmed Souabni$^a$
		\end{center}
		\vskip 0.5cm {\small
		\noindent $^a$ Carthage University,
		Faculty of Sciences of Bizerte, Department of  Mathematics, Jarzouna, 7021, Tunisia.}
	\begin{abstract}
		In this paper, we study the $L^p$-Bochner-Riesz mean summability problem related to the spectrum of some particular Sturm-Liouville operators in the weighted $L^p([a,b],\omega).$ Our purpose is to establish suitable conditions under which the Bochner-Riesz expansion of a function $f \in L^p([a,b],\omega)$,$1<p<\infty$, in two generalisations of Slepian's basis, converges to $f$ in $L^p([a,b],\omega)$.	
	\end{abstract}
	\vspace{5mm}
	
	\section{Introduction}
	The $L^p$-Bochner-Riesz mean convergence of orthogonal series has attracted special attention since several decades ago. This kind of convergence is briefly described as follows. Let 
	$1 \leq p <\infty$ , $a,b \in \mathbb{R}$ and $\{\varphi_n\}$ an orthonormal set of eigenfunctions of a positive self-adjoint differential operator $\mathcal{L} $ associated with eigenvalues $ \chi_n $ on a weighted Hilbert space 
	$L^2(I, \omega)$, where  $\omega$ is a positive bounded weight function. We define the expansion coefficients of $f\in L^p([a,b],\omega)$ by $ a_n(f) = \int_a^b f(x)\varphi_n(x) \omega(x)dx .$
	The orthonormal set $\{\varphi_n\}$ is said to have the Bochner-Riesz mean convergence of order $p$ over the Banach space
	$L^p(I,\omega)$ if for some suitable $\delta>0$ and for all $f \in L^p(I,\omega),$ we have
	\begin{equation}\label{eq-1}
	\lim_{R \to \infty} \int_{a}^{b}|f(x)- \Psi_R^{\delta}f(x) |^p \omega (x) dx =0,\mbox{  where  }\displaystyle\Psi_R^{\delta}f = \sum_{n=0}^{\infty} \Big(1-\frac{\chi_n}{R}\Big)^\delta_+ a_n(f) \varphi_n.
	\end{equation}
	To the best of our knowledge, M. Riesz was the first, in 1911,  to investigate this problem in some special cases.  Our problem is a modified summability method of Riesz mean introduced by Salomon Bochner given by \eqref{eq-1} . In \cite{Boc}, S.Bochner started by studying this problem for the trigonometric exponential case in higher dimension. Furthermore, in \cite{C.G}, the authors have proved a Bochner-Riesz mean convergence for the orthonormal eigenvectors system of a second order elliptic differential operator on a compact N-dimensional manifold M for $1\leq p\leq 2\frac{N+1}{N+3}$ and $\delta>N\left|\frac{1}{p}-\frac{1}{2}\right|-\frac{1}{2}$. Mauceri and Müller have also studied this problem in \cite{Mau} and \cite{Mul}  in the framework of the Heisenberg group. This problem has been analysed for Fourier-Bessel expansions series  in \cite{CIAl} and \cite{CIA2}. Moreover, in \cite{Casarino-M.Peloso}, authors have also solved this question in the case of sublaplacien on the sphere $S^{2n-1}$ in the complex n-dimensional space $\mathbb{C}^n$,where it has been shown that we have convergence for $ \delta > (2n-1)\Big|\frac{1}{2}-\frac{1}{p}\Big|$. The weak type convergence is investigated in this problem. Indeed, we say that an orthonormal family $\{\varphi_n\}$ of $L^p(I,\omega)$ has a weakly Bochner-Riesz mean convergence if $\Psi_R^{\delta}f$ converge to $f$  almost everywhere for every $f\in L^p(I,\omega)$. This problem has been solved in some special cases of orthonormal systems like Jacobi and Laguerre polynomials  in \cite{C.M} and for the eigenfunctions of the Hermite operator in higher dimension in \cite{C.D.H.L.Y}.\\ 
	In this work, we extend the $L^p$-Bochner-Riesz means convergence to  the circular and the generalized (or weighted) prolate spheroidal wave functions denoted by (CPSWFs) and (GPSWFs), respectively. The two last families are defined respectively as the eigenfunctions of the  operators $$ \mathcal H_c^{\alpha}f(x)=\int_0^1
	\sqrt{cxy}J_{\alpha}(cxy) f(y)dy,\quad \mathcal F_c^{(\alpha)} f(x)=\int_{-1}^1 e^{icxy}  f(y) (1-y^2)^{\alpha}\, dy,$$  where $\alpha>-1/2, \, c>0$ are two real numbers. These two sets of orthonormal functions are characterized as solutions of some Sturm-Liouville problems. The second family we consider is the weighted, some times called generalized, prolate spheroidal wave functions (GPSWFs) introduced by Wang-Zhang \cite{Wang2}. Note that the classical PSWFs  correspond to the special case of the GPSWFs with $\alpha=0.$
	
	Our aim in this paper is to prove the $L^p$convergence of Bochner-Riesz mean expansion in the GPSWFs and CPSWFs bases.  \\
	
	This work is organised as follows. In section 2, we give some mathematical preliminaries on Sturm-Liouville theory and some properties of the CPSWfs and GPSWFs. Note that these functions can be considered as generalizations of the spherical Bessel functions $j_{n}^{(\alpha)}$ and Gegenbauer's polynomials $\widetilde P_n^{(\alpha)}$, respectively. In section 3, we state our two main theorems and section 4 and 5 are respectively devoted to the proof of sufficient and necessary conditions of the main results.
	
	\section{Mathematical preliminaries}
	In this paragraph, we give some mathematical preliminaries that will be frequently used in the proofs of the different results of this work.
	\subsection{Some facts about Sturm-Liouville theory}
	The Sturm-Liouville differential  operator is defined as follows, see for example \cite{Sturm-Liouville Theory Past and Present},
	\begin{equation}\label{Sturm1}
	\mathcal{L} y(x)= \frac{d}{dx}[p(x) y'(x)] + q(x) y(x),\quad x\in I=(a,b).
	\end{equation}
	with $r=\frac{1}{p}, q \in L^1(I,\mathbb R).$
	The Sturm-Liouville eigenvalues problem is given by the following differential equation :
	\begin{equation}\label{Sturm2}
	\mathcal{L}.u(x)= -\chi \omega(x) u(x),\quad \sigma \in L^1(I,\mathbb R).
	\end{equation}
	That is
	\begin{equation} \label{S-Lequation}
	\frac{d}{dx}\Big[p(x)\frac{du}{dx} \Big]+q(x)u(x)+\chi\omega(x)u(x)=0,\quad x\in I.
	\end{equation}
	Note that a Sturm-Liouville operator satisfies the following properties,
	
	\begin{enumerate}
		\item $ u \mathcal{L} v-v \mathcal{L} u = \Big[ p(uv'-vu') \Big]' $ ( Lagrange's identity )
		\item The eigenvalues of $\mathcal{L}$ are real and form an infinite countable set $ \chi_0< \chi_1< \cdots <\chi_n <\cdots $ with $ \lim_{n\rightarrow +\infty} \chi_n = +\infty.$
		\item For each eigenvalue $ \chi_n$ there exists an eigenfunction $ \phi_n $ having n zeros on $ [ a,b ]. $
		\item Eigenfunctions corresponding to different eigenvalues are orthogonal with respect to the following inner product
		\begin{equation*}
		\scal{f,g}_{\omega} = \int_a^b f(x)g(x)\omega(x)dx,\quad f, g \in L^2(I, \omega).
		\end{equation*}
		In the sequel, we assume that $\omega(x)\geq 0$, for $x\in (a,b).$
	\end{enumerate}

	\subsection{Some facts about GPSWFs and CPSWFs}	
	We first recall that, for $c>0$, the prolate spheroidal wave functions PSWFs, denoted  $\psi_{n,c} $, have been introduced by D.Slepian as solutions of the following energy maximization problem 
	\begin{equation*}
	\mbox{ Find } f=\arg\max_{f\in B_c}\frac{\int_{-1}^1|f(t)|^2 dt}{\int_{\mathbb{R}}|f(t)|^2 dt},
	\end{equation*}
	where $B_c$ is the classical Paley-Wiener space, defined by
	\begin{equation}
	\label{Bc}
	B_c=\left \{ f\in L^2(\mathbb R),\,\, \mbox{Support } \widehat f\subseteq [-c,c] \right\}.
	\end{equation}
	Here, $\widehat f$ is the Fourier transform of $f\in L^2(\mathbb R).$
	It has been shown that they are also eigenfunctions of the integral operator with sinc kernel. A breakthrough in the theory of Slepian functions is due to Slepian,Pollard and Landau who have proved that PSWFs are also eigenfunctions of a Sturm-Liouville operator by proving a commutativity property. For more details about Slepian's functions we refer reader to \cite{Slepian1,Slepian2,Slepian3}.
	In this work we are interested in two generalizations of the PSWFs.\\
	The first basis is called circular prolate spheroidal wave functions (CPSWFs) or radial part of the 2d-Slepian, introduced by D.Slepian\cite{Slepian3} as solutions of the following problem 
	\begin{equation*}
	\mbox{ Find } f=\arg\max_{f\in HB^\alpha_c}\frac{\int_{0}^1|f(t)|^2 dt}{\int_{0}^\infty|f(t)|^2 dt},
	\end{equation*}
	where $HB^\alpha_c$ is the Hankel Paley-Wiener space, defined by
	\begin{equation}
	\label{HBc}
	HB^\alpha_c=\left \{ f\in L^2(\mathbb R),\,\, \mbox{Support } \mathcal{H}^{\alpha}f\subseteq [-c,c] \right\}.
	\end{equation}
	Here the Hankel transform $ \mathcal{H}^\alpha$ is defined, for $ f\in L^1(0,\infty)$, by
	$$ \mathcal{H}^{\alpha}f(x)=\int_{0}^{\infty}\sqrt{xy}J_\alpha(xy)f(y)dy.$$ 
	Here $J_\alpha(.)$ is the Bessel function and $ \alpha>-1/2$. Like Fourier transform, $\mathcal{H}^\alpha$ can be extended into a unitary operator on $L^2(0,\infty)$.
	They are also the different band-limited eigenfunctions of the finite Hankel transform $\mathcal{H}_c^{\alpha}$ defined on $L^2(0,1)$ with kernel $H_c^{\alpha}(x,y)=\sqrt{cxy}J_{\alpha}(cxy)$ where $J_{\alpha}$ is the Bessel function of the first type and order $\alpha>-\frac{1}{2}$(see for example \cite{Slepian3}). That is 
	\begin{equation}
	\mathcal{H}_c^{\alpha}(\varphi^{\alpha}_{n,c})=\mu_{n,\alpha}(c)\varphi^{\alpha}_{n,c}.
	\end{equation}
	
	In his pioneer work \cite{Slepian3}, D. Slepian has shown that the compact integral operator $\mathcal{H}_c^{\alpha}$ commutes with the following Sturm-Liouville differential operator $\mathcal L^{\alpha}_c$ defined on $C^2([0,1])$ by
	
	\begin{equation}\label{differ_operator1}
	\mathcal{L}_c^{\alpha}(\phi)=- \dfrac{d}{dx} \left[ (1-x^2)\dfrac{d}{dx} \phi \right]+\left(c^2x^2-\dfrac{\dfrac{1}{4}-\alpha^2}{x^2}\right)\phi.
	\end{equation}
	
	Hence,  $\varphi^{\alpha}_{n,c}$ is the $n-$th bounded   eigenfunction of the positive self-adjoint operator $\mathcal{L}_c^{\alpha}$  associated with the eigenvalue $\chi_{n,\alpha}(c),$ that is
	\begin{equation}
	\label{differ_operator2}
	-\dfrac{d}{dx} \left[ (1-x^2)\dfrac{d}{dx} \varphi^{\alpha}_{n,c}(x) \right] + \left(c^2x^2- \dfrac{\dfrac{1}{4}-\alpha^2}{x^2} \right)\varphi^{\alpha}_{n,c}(x)=\chi_{n,\alpha}(c)\varphi^{\alpha}_{n,c}(x),\quad x\in [0,1].
	\end{equation}
	The orthonormal family $\varphi_{n,c}^{\alpha}$ form an orthonormal basis of $L^2(0,1)$ and the associated eigenvalues family $\chi_{n,\alpha}(c)$ satisfy the following inequality, (see \cite{Slepian3})
	\begin{equation}\label{boundschi2}
	(2n+\alpha+1/2)(2n+\alpha+3/2)\leq\chi_{n,\alpha}(c)\leq (2n+\alpha+1/2)(2n+\alpha+3/2)+c^2
	\end{equation} 
	The second family we consider in this work is the weighted, (some times called generalized), prolate spheroidal wave functions introduced by Wang-Zhang \cite{Wang2} as solutions of a Sturm-Liouville problem or equivalently eigenfunctions of an integral operator. GPSWFs are also solutions of the following problem as given in \cite{Karoui-Souabni1}
	\begin{equation*}
	\mbox{Find } f= {\displaystyle arg\max_{f\in B^{\alpha}_c}  \frac{\|f\|^2_{L^2_{\omega_{\alpha}}(I)}}{\|\widehat f\|^2_{L^2(\omega_{-\alpha}(\frac{\cdot}{c}))}} },
	\end{equation*}
	where $\omega_\alpha(x) = (1-x^2)^\alpha$ and $B^{(\alpha)}_c$ is the restricted Paley-Winer space, defined by
	\begin{equation*}
	B_c^{(\alpha)}=\{ f\in L^2(\mathbb R),\,\,  \mbox{Support } \widehat f\subseteq [-c,c],\, \, \widehat f\in L^2\big((-c,c), \omega_{- \alpha}(\frac{\cdot}{c})\big)\}.
	\end{equation*}
	More precisely, the GPSWFs are the eigenfunctions of the weighted finite Fourier transform operator
	$\mathcal F_c^{(\alpha)}$ defined  by
	\begin{equation}\label{Eq1.1}
	\mathcal F_c^{(\alpha)} f(x)=\int_{-1}^1 e^{icxy}  f(y)\,\omega_{\alpha}(y)\,\mathrm{d}y.
	\end{equation}
	It is well known, (see \cite{Karoui-Souabni1, Wang2}) that they are also eigenfunctions of the compact and positive operator
	$$
	\mathcal Q_c^{(\alpha)}=\frac{c}{2\pi}
	\mathcal F_c^{({\alpha})^*} \circ \mathcal F_c^{(\alpha)}
	$$
	which is  defined on $L^2(I,\omega_{\alpha})$ by
	\begin{equation}\label{EEq0}
	\mathcal Q_c^{(\alpha)} g (x) = \int_{-1}^1 \frac{c}{2 \pi}\mathcal K_{\alpha}(c(x-y)) g(y) \omega_{\alpha}(y) \d y
	\end{equation}
	Here,
	$$
	\mathcal K_{\alpha}(x)=\sqrt{\pi} 2^{\alpha+1/2}\Gamma(\alpha+1) \frac{J_{\alpha+1/2}(x)}{x^{\alpha+1/2}}
	$$
	
	
	It has been shown in \cite{Karoui-Souabni1, Wang2} that the last two integral operators commute with the following
	Sturm-Liouville operator $\mathcal L_c^{(\alpha)}$ defined on $C^2[-1,1]$ by
	\begin{equation}\label{diff_oper1}
	\mathcal L_c^{(\alpha)} (f)(x)= - \frac{1}{\omega_{\alpha}(x)} \frac{d}{dx}\left[ \omega_{\alpha}(x) (1-x^2) f'(x)\right] +c^2 x^2  f(x).
	\end{equation}
	Also, note that  the $(n+1)-$th eigenvalue $\chi_{n,\alpha}(c)$ of $\mathcal L_c^{(\alpha)}$ satisfies the following classical inequalities,
	\begin{equation}
	\label{boundschi}
	n(n+2\alpha+1) \leq \chi_{n,\alpha}(c) \leq n (n+2\alpha+1) +c^2,\quad \forall n\geq 0.
	\end{equation}
	
	\section{Statement of results}
	In this section, we will state the main results of this paper that we will prove in the following sections. As mentioned before, the main issue studied in this paper is to get a necessary and sufficient conditions of Bochner-Riesz expansion convergence of a function $f$ in the GPSWFs's and CPSWFs's basis. Let's start by studying the case of GPSWFs in the following theorem.
	\begin{theorem}
		Let $0\leq\alpha<3/2$, $\delta$ and $c$ be two positive number and $(\psi_{n,c}^{(\alpha)})_{n\geq 0}$ be the family of weighted prolate spheroidal wave functions.
		For a smooth function $f$ on $I=(-1,1)$, we define
		$$
		\Psi_R^{\delta}f=\sum_{n=0}^{\infty}\left(1-\frac{\chi_{n,\alpha}(c)}{R}\right)^{\delta}_+\scal{f,\psi_{n,c}^{(\alpha)}}_{L^2(I,\omega_\alpha)}\psi_{n,c}^{(\alpha)}.
		$$
		Then, for every $1\leq p<\infty$, $\Psi^{\delta}_R$ can be extended to a bounded operator $L^p(I,\omega_\alpha)\to L^p(I,\omega_\alpha)$.
		Further, $\Psi^{\delta}_Rf$ is uniformly bounded if ,and only if,  $\delta>\max\{\frac{\gamma_{\alpha}(p')}{2},0\}$ and $p\not=p_0=2-\frac{1}{\alpha+3/2}$ where $$\gamma_{\alpha}(p)=\begin{cases}0 &\mbox{ if } 1<p<p'_0\\ \epsilon &\mbox{ if } p=p'_0\\ 2(\alpha+1)\left[\frac{1}{2}-\frac{1}{p}\right]-\frac{1}{2}&\mbox{ if } p>p'_0\\\alpha+1&\mbox{ if } p=1 \end{cases}.$$ 
		and $\epsilon$ is an arbitrary real number.	
		Note that $p'$ denote here the dual exponent of $p$.
	\end{theorem}
\begin{remark}
	The sufficient condition in both GPSWFs's and CPSWFs's case still valid even for all $\alpha>-1/2$.
\end{remark}
\begin{remark}[Two special cases]
	
	Recall that $\psi^{(\alpha)}_{n,0} = \wJ^{(\alpha,\alpha)}_n$, then we recover the same result for the case of normalized Geganbauer polynomials. Note that both conditions (A) and (B), defined in the proof of the last theorem, are still valid even for $\wJ_n^{(\alpha,\beta)}$ with exactly the same proof and by noticing that the transferring theorem, which is the key step of the necessary condition, has been proven in \cite{S.I} , the last result is valid also for Jacobi polynomials for all $\alpha, \beta >-1/2$ . For $\alpha=0$ and $c>0$, $\psi^{0}_{n,c} =\psi_{n,c}$ presents the classical prolate spheroidal wave functions PSWFs which satisfy the Bochner-Riez mean convergence if and only if $\delta> \max\{0,\frac{\gamma_0(p')}{2}\}.$
\end{remark}
Let us now focus on the circular case.
\begin{theorem}
	Let $\alpha\geq 1/2$, $c>0$ and $(\varphi_{n,c}^{(\alpha)})_{n\geq 0}$ be the family of Hankel prolate spheroidal wave functions.
	For a smooth function $f$ on $I=(0,1)$, we define
	$$
	\Psi_R^{\delta}f=\sum_{n=0}^{\infty}\left(1-\frac{\chi_{n,\alpha}(c)}{R}\right)^{\delta}_+\scal{f,\varphi_{n,c}^{(\alpha)}}_{L^2(0,1)}\varphi_{n,c}^{(\alpha)}.
	$$
	Then, for every $1\leq p<\infty$, $\Psi^{\delta}_R$ can be extended to a bounded operator $L^p(0,1)\to L^p(0,1)$.
	Further
	$
	\Psi^{\delta}_Rf
	$
	is uniformly bounded if ,and only if, $\delta>\max\{\frac{\gamma(p')}{2},0\}$ where 
	$$\gamma(p)=\begin{cases}\frac{1}{p}-\frac{1}{2} &\mbox{ if } 1<p<4\\ \epsilon-\frac{1}{4} &\mbox{ if } p=4\\ \frac{1}{3}\left[\frac{1}{p}-1\right]&\mbox{ if } p>4\\ 1&\mbox{ if } p=1 \end{cases}.$$.
\end{theorem}
	
	\section{Proof of sufficient condition}
		Let $(I,\omega)$ be a measured space such that $\omega$ is a bounded weight function. We denote by $p' = \frac{p}{p-1}$ the dual index of $p$.
	
	Throughout this section, $\mathcal{L}$ denotes a Sturm-Liouville operator and $\varphi_n$ (respectively $\lambda_n$) the sequence of the associated eigenfunctions (respectively eigenvalues). The Riesz means of index $\delta >0 $ associated with $\mathcal{L}$ of a function $f \in \mathcal{C}^\infty(I,\R) $ are consequently defined as 
	\begin{equation}
	\Psi^{\delta}_Rf = \sum_{n=0}^{\infty}\Big(1-\frac{\lambda_n}{R}\Big)^{\delta}_+ a_n(f) \varphi_n \quad \mbox{with} \quad a_n(f) = \int_I f(y) \varphi_n(y) d\mu(y). 
	\end{equation}
	$\Psi^{\delta}_Rf$ can also be written as 
	$$ \Psi^{\delta}_R.f(x) = \int_{I} K_R^\delta(x,y)f(y) d\mu(y) \quad \mbox{where} \quad 
	K_R^\delta(x,y) = \sum_{n=0}^{\infty} \Big(1-\frac{\lambda_n}{R}\Big)^{\delta}_+ \varphi_n(x) \varphi_n(y) $$
	
	Our aim in this section is to prove at the same time the sufficient conditions of our two main theorems. 
	More precisely, we will define several conditions on $\varphi_n$ that will ensure the convergence of $\Psi_R^{\delta}.f$ to $f$ in the $L^p$norm as $R \to \infty $ and verify that both two families satisfy these conditions. Assume that $\varphi_n$ satisfies the following conditions :
	
	\begin{itemize}
		\item[$(A)$] For every $1\leq p\leq\infty$, every $n$, $\varphi_n\in L^p(I,\omega)$.
		Further, we assume that there is a constant $\gamma(p)\geq 0$ such that $\norm{\varphi_n}_{L^{p}(\mu)}\leq Cn^{\gamma(p)}$.
		\item[$(B)$] The sequence $(\lambda_n)$ of the eigenvalues of the operator $\mathcal{L}$ satisfies the following properties
		\begin{enumerate}
			\item $\displaystyle\sum_{\lambda_n\in (m,M)} 1\leq C(M-m)$
			for all $0\leq m<M$.\\
			\item There exists $\varepsilon>0$ such that 
			$$\lambda_n\geq Cn^{\varepsilon}.$$
		\end{enumerate}
	\end{itemize}
	
	First of all, we start by giving sense to $\Psi^{\delta}_R.f$ for every $f \in L^p(\mu) $. Indeed, 
	$$ \norm{K^{\delta}_R}_{L^p(\mu)\otimes L^{p'}(\mu)} \leq \sum_{\lambda_n<R} \norm{\varphi_n}_p \norm{\varphi_n}_{p'} \leq \sum_{\lambda_n<R}n^{\gamma(p)+\gamma(p')} \leq CR^{\frac{\left(\gamma(p)+\gamma(p')\right)}{\varepsilon}+1} , $$
	So that the integral operator $\Psi_R^\delta $ can be extended to a continuous operator $L^p(\mu) \to L^p(\mu) $ with $$ \norm{\Psi_R^{\delta}}_{L^p\to L^p} \leq \norm{K^{\delta}_R}_{L^p(\mu)\otimes L^{p'}(\mu)} .$$
	The following theorem is one of the main results of this paper.
	\begin{theorem}
		With the above notation and under conditions $(A)$ and $(B)$ with $\delta>\delta(p)=\max\{\frac{\gamma(p')}{\varepsilon},0\}$, there exists a constant $C>0$ satisfying the following inequality
		\begin{equation} \label{lpbound}
		\norm{\Psi_R^{\delta}}_{(L^p(I,w),L^p(I,w))}\leq C.
		\end{equation}
	\end{theorem}
	The following lemma will be used in the proof of the previous theorem.  
	\begin{lemma}
		Let $1 \leq p\leq 2$ then for every $f\in L^p(I,\omega),$ we have
		\begin{equation}
		\norm{\sum_{\lambda_n\in (m,M)}a_n(f)\varphi_n}_{L^2(I,\omega)}\leq C(p) M^{\frac{\gamma(p')}{\varepsilon}}(M-m)^{\frac{1}{2}}\norm{f}_{L^p(I,\omega)}.
		\end{equation}
	\end{lemma}
	\begin{proof}
		By orthogonality
		and Hölder's inequality, we have
		\begin{eqnarray*}
			\norm{\sum_{\lambda_n\in (m,M)}a_n(f)\varphi_n}^2_{L^2(I,\omega)}&=&\sum_{\lambda_n\in (m,M)}a_n^2(f) \leq \sum_{\lambda_n\in (m,M)}\norm{\varphi_n}^2_{L^{p'}(I,\omega)}\norm{f}^2_{L^p(I,\omega)}	
		\end{eqnarray*}
		From condition $(A)$, we have
		$\norm{\varphi_n}_{L^{p'}(I,\omega)}\leq n^{\gamma(p')}$. We also obtain by using condition ($B1$)
		
		\begin{eqnarray*}
			\norm{\sum_{\lambda_n\in (m,M)}a_n(f)\varphi_n}^2_{L^2(I,\omega)}&\leq&\sum_{\lambda_n\in (m,M)} n^{2\gamma(p')}\norm{f}^2_{L^p(I,\omega)} \leq C\sum_{\lambda_n\in (m,M)} \lambda_n^{\frac{2\gamma(p')}{\varepsilon}}\norm{f}^2_{L^p(I,\omega)}\\&\leq&CM^{\frac{2\gamma(p')}{\varepsilon}}\left(\sum_{\lambda_n\in (m,M)}1\right)\norm{f}^2_{L^p(I,\omega)}\\&\leq&CM^{(\frac{2\gamma(p')}{\varepsilon})}(M-m) \norm{f}^2_{L^p(I,\omega)}.
		\end{eqnarray*}
		Then one gets
		$$	\norm{\sum_{\lambda_n\in (m,M)}a_n(f)\varphi_n}_{L^2(I,\omega)}\leq C M^{(\frac{\gamma(p')}{\varepsilon})}(M-m)^{\frac{1}{2}} \norm{f}_{L^p(I,\omega)}.$$
		
	\end{proof}
	\begin{proof}[Proof of Theorem1]

		We should mention here that some parts of the proof of this theorem are inspired from  \cite{Casarino-M.Peloso}.
		Without loss of generality, we can consider $1\leq p<2$ and conclude by duality.
		To prove \eqref{lpbound}, we are going to have to decompose the multiplier $\Psi_{R}^\delta$. In order to do so, let $\phi \in\mathcal{C}^{\infty}_0(\mathbb{R})$ with support on $(1/2,2)$ such that $\displaystyle\sum_{k\in\mathbb{Z}}\phi(2^kt)=1$ and $\displaystyle \phi_0(t)=1-\sum_{k=1}^{+\infty}\phi(2^kt)$ for all $t>0$. We define $$\phi_{R,k}^{\delta}(t)=\left(1-\frac{t}{R}\right)_+^{\delta}\phi\left(2^k(1-\frac{t}{R})\right).$$
		We recall that, from \cite{Casarino-M.Peloso}, this last function has the following properties :
		\begin{enumerate}
			\item $\mbox{supp}\left(\phi_{R,k}^{\delta}\right)\subseteq (R(1-2^{-k+1}),R(1-2^{-k-1}))$, \\
			\item $\sup_{t\in\mathbb{R}}|\phi_{R,k}^{\delta}(t)|\leq C2^{-k\delta}$,\\
			\item $\forall N\geq 0,$ there exists $C_N>0$ such that
			$$|\partial_t^N\phi_{R,k}^{\delta}(t)|\leq C_N\Big(\frac{2^k}{R}\Big)^N. $$ 
		\end{enumerate}
		Furthermore, we denote by
		\begin{equation}
		\Psi_{R,k}^{\delta}.f=\sum_{n=0}^{\infty}\phi_{R,k}^{\delta}(\lambda_n)a_n(f)\varphi_n \qquad k=1,2,\cdots
		\end{equation}
		Then, we have
		\begin{eqnarray*}
			\Psi_R^{\delta}f&=&\sum_{n=0}^{\infty}\left(1-\frac{\lambda_n}{R}\right)_+^{\delta}a_n(f)\varphi_n\\&=&\sum_{n=0}^{\infty}\phi_0(1-\frac{\lambda_n}{R})\left(1-\frac{\lambda_n}{R}\right)_+^{\delta}a_n(f)\varphi_n+\sum_{n=0}^{\infty}\left(\sum_{k=1}^{+\infty}\phi(2^k(1-\frac{\lambda_n}{R}))\right)\left(1-\frac{\lambda_n}{R}\right)_+^{\delta}a_n(f)\varphi_n\\&=&\sum_{n=0}^{\infty}\phi_0(1-\frac{\lambda_n}{R})\left(1-\frac{\lambda_n}{R}\right)_+^{\delta}a_n(f)\varphi_n+\sum_{k=1}^{\left[\frac{\log(R)}{\log(2)}\right]}\sum_{n=0}^{\infty}\phi_{R,k}^{\delta}(\lambda_n)a_n(f)\varphi_n+\sum_{k=\left[\frac{\log(R)}{\log(2)}\right]+1}^{\infty}\sum_{n=0}^{\infty}\phi_{R,k}^{\delta}(\lambda_n)a_n(f)\varphi_n\\&=&\psi_{R,0}^{\delta}f+\sum_{k=1}^{\left[\frac{\log(R)}{\log(2)}\right]}\Psi_{R,k}^{\delta}f+\mathcal{R}_{R}^{\delta}f.
		\end{eqnarray*}
		
		It is clear that the main term is the second one. With the same approach used in \cite{Casarino-M.Peloso}, we will prove the following proposition :
		
		\begin{proposition}
			Let $1\leq p< 2$ and $\delta>\delta(p)=\frac{\gamma(p')}{\varepsilon}$. There exists $\beta>0$ such that for every $f\in L^p(I,w)$, we have
			\begin{equation}\label{lpineq}
			\norm{\Psi_{R,k}^{\delta}f}_{L^p(I,w)}\leq C2^{-k\beta}\norm{f}_{L^p(I,w)},
			\end{equation}
			where $C$ is a constant independent of $R$ and $f$.
		\end{proposition}
		\begin{proof}
			Let $x_0=\frac{a+b}{2}\in (a,b)$ and $r=\frac{b-a}{4}>0$ such that $(x_0-r,x_0+r)\subseteq (a,b).$ Note that, for every $1\leq k\leq \left[\frac{\log(R)}{\log(2)}\right]=k_R$, we have  $r_{k}^{\alpha}=\left(\frac{2^{k}}{R}\right)^{\mu(p)}r< r$ where $\mu(p)=\frac{(\frac{\gamma(p')}{\varepsilon}+\frac{1}{2})}{(\frac{1}{p}-\frac{1}{2})}$ . So we notice that $I=(a,b)=(x_0-r_{k}^{\alpha},x_0+r_{k}^{\alpha})\cup\{y\in (a,b), |y-x_0|>r_{k}^{\alpha}\}=I_{k,1}^{\alpha}\cup I_{k,2}^{\alpha}$.\\
			
			We start by providing an $L^p$ bound of $\norm{\Psi_{R,k}^\delta}_{L^p(I^\alpha_{k,1},\omega)}$. To do so, we proceed in the way to reduce the $L^p$ inequality \eqref{lpineq} to certain $(L^p,L^2)$ inequality using the last lemma.
			
			Using Parseval formula and the fact that $\mbox{supp}\left(\phi_{R,k}^{\delta}\right)\subseteq (R_{k,1},R_{k,2})$,	where $R_{k,1}=R(1-2^{-k+1})$ and $R_{k,2}=R(1-2^{-k-1}),$ we have
			\begin{eqnarray*}
				\norm{\Psi_{R,k}^{\delta}f}^2_{L^2(I,w)}&=&\norm{\sum_{n=0}^{\infty}\phi_{R,k}^{\delta}(\lambda_n)a_n(f)\varphi_n}^2_{L^2(I,w)}\\&=&\norm{\sum_{R_{k,1}\leq\lambda_n\leq R_{k,2}}\phi_{R,k}^{\delta}(\lambda_n)a_n(f)\varphi_n}^2_{L^2(I,w)},
			\end{eqnarray*}
			Using the previous lemma with $m=R_{k,1}$, $M=R_{k,2}$ and the fact that $\displaystyle \sup_{t\in\mathbb{R}}|\phi_{R,k}^{\delta}(t)|\leq C2^{-k\delta}$, one gets
			\begin{eqnarray*}
				\norm{\Psi_{R,k}^{\delta}f}^2_{L^2(I,w)}&\leq&C2^{-2k\delta}\norm{\sum_{R_{k,1}\leq\lambda_n\leq R_{k,2}}a_n(f)\varphi_n}^2_{L^2(I,w)}\\&\leq&C2^{-2k\delta}R^{(2\frac{\gamma(p')}{\varepsilon})}\left(\frac{3R}{2^{k+1}}\right)\norm{f}^2_{L^p(I,w)}.	
			\end{eqnarray*}
			Hence, we have 
			\begin{equation}\label{eq0}
			\norm{\Psi_{R,k}^{\delta}f}_{L^2(I,w)}\leq C2^{-k(\delta+\frac{1}{2})}R^{(\frac{\gamma(p')}{\varepsilon})+\frac{1}{2})}\norm{f}_{L^p(I,w)}.
			\end{equation}
			
			By combining Hölder inequality and \eqref{eq0}, we obtain
			\begin{eqnarray}\label{eq6}
			\norm{\Psi_{R,k}^{\delta}f}_{L^p(I_{k,1}^{\alpha},w)}&\leq& (\mu(I_{k,1}))^{\frac{1}{p}-\frac{1}{2}}\norm{\Psi_{R,k}^{\delta}f}_{L^2(I_{k,1}^{\alpha},w)}\nonumber\\&\leq&(2r_{k}^{\alpha})^{\frac{1}{p}-\frac{1}{2}}\norm{\Psi_{R,k}^{\delta}f}_{L^2(I,w)}\nonumber\\
			&\leq&C2^{-k(\delta-\frac{\gamma(p')}{\varepsilon})}\norm{f}_{L^p(I,w)}.
			\end{eqnarray}
			
			Let $\displaystyle s_{R,k}^{\delta}(u,v)=\sum_{n=0}^{\infty}\phi_{R,k}^{\delta}(\lambda_n)\varphi_n(x)\varphi_n(y)$ be the kernel of $\Psi_{R,k}^{\delta}$. We just have to find an estimate of $||\Psi_{R,k}^{\delta}f||_{L^p(I_{k,2}^{\alpha},w)}$, so we will use the Schur test with the symmetric property of $s_{R,k}^{\delta},$  then it suffices to prove the following inequality
			$$\sup_{u\in I_{k,2}^{\alpha}}\norm{s_{R,k}^{\delta}(u,.)}_{L^1(I_{k,2}^{\alpha})}\leq C2^{-k\varepsilon}$$
			for some $\varepsilon>0$ and $C>0$ depending only on $p.$\\
			We consider $g_{R,k}^{\delta}(\lambda)=\left(1-\frac{\lambda^2}{R}\right)_+^{\delta}e^{\lambda^2/R}\phi(2^k(1-\frac{\lambda^2}{R}))$ satisfying the following properties, see \cite{Casarino-M.Peloso}
			\begin{enumerate}
				\item For every non-negative integer $i$ there exists a constant $C_i$ such that for all $s>0$
				\begin{equation}\label{eq2}
				\int_{|t|\geq s}|\hat{g}_{R,k}^{\delta}(t)|dt\leq C_is^{-i}R^{-i/2}2^{(i-\delta)k}
				\end{equation}
				\item
				\begin{equation}\label{eq4}
				\norm{g_{R,k}^{\delta}(\sqrt{\mathcal{L}})}_{(L^2,L^2)}\leq C2^{-k\delta}.
				\end{equation}
				
			\end{enumerate}
			
			For our purpose, we will consider such a positive self-adjoint operator $\mathcal{L}$ on $L^2(\mathbb{R})$ such that the semigroup $e^{-t\mathcal{L}}$, generated by $-\mathcal{L}$, has the kernel $p_t(x, y)$ obeying the Gaussian upper bound 
			\begin{equation}\label{eq7}
			|p_t(u,v)|\leq \frac{C}{\sqrt{t}}\exp{\left(-\frac{|u-v|^2}{Ct}\right)}.
			\end{equation}
			for a constant $C > 0$. (see \cite{E.B. Davies}) \\
			For all $ u \in \mathbb{R}$ and $t>0$, one gets the following estimate 
			\begin{equation}\label{eq3}
			\norm{p_t(u, .)}_{L^2(\mathbb{R})}\leq C.
			\end{equation}
			
			On the other hand, there exists $i_0\in \mathbb{N}$ such that $2^{i_o-1}<R^{\mu(p)}<2^{i_0}$ and we can see that $$I_{k,2}^{\alpha}\subseteq\displaystyle\cup_{\mu(p)k-i_0\leq j\leq 0}D_j$$ where $D_j=\{y, 2^jr\leq |y-x_0|<2^{j+1}r\}.$
			Since, $\mathcal{L}$ is a positive self-adjoint operator, then it's clear that
			\begin{equation}\label{eq1}
			\phi_{R,k}^{\delta}(\mathcal{L})=g_{R,k}^{\delta}(\sqrt{\mathcal{L}})\exp{\left(-\mathcal{L}/R\right)}.
			\end{equation}
			Hence one gets
			\begin{eqnarray*}
				s_{R,k}^{\delta}(u,v)&=&g_{R,k}^{\delta}(\sqrt{\mathcal{L}})\left(p_{1/R}(u, .)\right)(v)\\&=&g_{R,k}^{\delta}(\sqrt{\mathcal{L}})\left(p_{1/R}(u, .)\chi_{\{w,|x_0-w|<2^{j-1}r\}}\right)(v)+g_{R,k}^{\delta}(\sqrt{\mathcal{L}})\left(p_{1/R}(u, .)\chi_{\{w,|x_0-w|\geq2^{j-1}r\}}\right)(v)\\&=&s_{R,k}^{\delta,1}(u,v)+s_{R,k}^{\delta,2}(u,v).
			\end{eqnarray*}
			Using the fact that $g_{R,k}^{\delta}$ is an even function, with the inversion formula, we have $$g_{R,k}^{\delta}(\sqrt{\lambda})=\frac{1}{\sqrt{2\pi}}\int_{\mathbb{R}}\hat{g}_{R,k}^{\delta}(t)\cos{(t\sqrt{\lambda})}dt.$$ Hence, we obtain
			\begin{eqnarray*}
				s_{R,k}^{\delta,1}(u,v)&=&\frac{1}{\sqrt{2\pi}}\int_{\mathbb{R}}\hat{g}_{R,k}^{\delta}(t)\cos{(t\sqrt{\mathcal{L}})}\left(p_{1/R}(u, .)\chi_{\{w,|x_0-w|<2^{j-1}r\}}\right)(v)dt.
			\end{eqnarray*}
			Moreover, the operator  $\cos{(t\sqrt{\mathcal{L}})}$ is bounded in $L^2$ with support kernel $\mathcal{K}_t$ satisfying, see \cite{E.B. Davies, L.Song-J.Xiao-X.Yan}
			$$\mbox{Supp}\left(\mathcal{K}_t\right)=\{(u,v)\in \mathbb{R}^2, |u-v|\leq c_0|t|\}$$ 
			From \eqref{eq4}, \eqref{eq3} and the previous analysis, one gets
			\begin{eqnarray*}
				\norm{s_{R,k}^{\delta,1}(u,.)}_{L^1(D_j)}&=&\frac{1}{\sqrt{2\pi}}\norm{\int_{\mathbb{R}}\hat{g}_{R,k}^{\delta}(t)\cos{(t\sqrt{\mathcal{L}})}\left(p_{1/R}(u, .)\chi_{\{w,|x_0-w|<2^{j-1}r\}}\right)dt}_{L^1(D_j)}\\&=&\frac{1}{\sqrt{2\pi}}\norm{\int_{|t|>\frac{2^{j-1}r}{c_0}}\hat{g}_{R,k}^{\delta}(t)\cos{(t\sqrt{\mathcal{L}})}\left(p_{1/R}(u, .)\chi_{\{w,|x_0-w|<2^{j-1}r\}}\right)dt}_{L^1(D_j)}\\&\leq&\frac{\mu^{1/2}(D_j)}{\sqrt{2\pi}} \int_{|t|>\frac{2^{j-1}r}{c_0}}|\hat{g}_{R,k}^{\delta}(t)|\norm{p_{1/R}(u, .)}_{L^2(D_j)}dt\\&\leq&\frac{C}{\sqrt{2\pi}}2^{\frac{j+1}{2}}\int_{|t|>\frac{2^{j-1}r}{c_0}}|\hat{g}_{R,k}^{\delta}(t)|dt 
			\end{eqnarray*}
			Let $i>\frac{\mu+\frac{1}{2}}{2(\mu+1-\frac{1}{p})}$ where $\mu=\frac{\gamma(p')}{\varepsilon}>0.$ Then by \eqref{eq2}, there exists a constant $C_i>0$ such that 
			\begin{eqnarray*}
				\norm{s_{R,k}^{\delta,1}(u,.)}_{L^1(D_j)}&\leq&\frac{C_i}{\sqrt{2\pi}}2^{j/2}(\frac{2^j}{c_0})^{-i}R^{-i/2}2^{(i-\delta)k}\\&\leq&\frac{C_i}{\sqrt{2\pi}}c_0^{i}2^{(i-\delta)k}2^{j(1/2-i)}.
			\end{eqnarray*}
			Then, we obtain
			\begin{eqnarray*}
				\norm{s_{R,k}^{\delta,1}(u,.)}_{L^1(I_{k,2})}&\leq&\sum_{\mu(p)k-i_0\leq j\leq 0}\norm{s_{R,k}^{\delta,1}(u,.)}_{L^1(D_j)}\\&\leq&\frac{C_i}{\sqrt{2\pi}}c_0^{i}2^{(i-\delta)k}\sum_{\mu(p)k-i_0\leq j\leq 0}2^{j(1/2-i)}\\&\leq&\frac{C_i}{\sqrt{2\pi}}c_0^{i}2^{(i-\delta)k}2^{(i-1/2)(i_0-\mu(p)k+1)}\\&\leq&C'_i2^{-k\varepsilon_1}.
			\end{eqnarray*}
			From our assumption on  $i,$ $\varepsilon_1=\delta-i+(i-1/2)(\frac{\mu+\frac{1}{2}}{(\frac{1}{p}-\frac{1}{2})}))>0.$ \\
			Then, to have an estimate of the kernel $s_{R,k}^{\delta,1}$ on $L^1(I_{k,2}^{\delta})$, it suffices to find an estimate of the kernel $s_{R,k}^{\delta,2}$ on $L^1(I_{k,2}^{\delta})$.\\
			From \eqref{eq4}, \eqref{eq7} and using the fact that $R\leq R^{\mu(p)}$, one gets the following inequality
			\begin{eqnarray*}
				\norm{s_{R,k}^{\delta,2}(u,.)}_{L^1(D_j)}&=&\int_{D_j}|g_{R,k}^{\delta}(\sqrt{\mathcal{L}})\left(p_{1/R}(u,.)\chi_{\{w,|w-x_0|>2^{j-1}r\}}\right)(v)|dv\\&\leq&\norm{g_{R,k}^{\delta}(\sqrt{\mathcal{L}})}_{(L^2,L^2)}\norm{p_{1/R}(u,.)\chi_{\{w,|w-x_0|>2^{j-1}r\}}}_{L^2(D_j)}\\&\leq& C2^{-k\delta}\sqrt{R}e^{(-CR2^{2j-2})}\left(\mu(D_j)\right)^{1/2}\\&\leq&C2^{-k\delta}2^{\frac{i_0+j}{2}}e^{-C2^{2(i_0+j)}}.
			\end{eqnarray*}
			Hence, we conclude that
			\begin{eqnarray*}
				\norm{s_{R,k}^{\delta,2}(u,.)}_{L^1(I_{k,2})}&\leq&\sum_{\mu(p)k-i_0\leq j\leq 0}\norm{s_{R,k}^{\delta,2}(u,.)}_{L^1(D_j)}\\&\leq&C2^{-k\delta}\sum_{i=i_o+j\geq\mu(p)k}2^{\frac{i}{2}}e^{-C2^{2i}}\\&\leq&C'2^{-k\delta}.
			\end{eqnarray*}
			
		\end{proof}
		
		\begin{proposition}
			Let $1\leq p\leq 2$ and $\delta>\delta(p)=\frac{\gamma(p')}{\varepsilon}$, then for all $f\in L^p(I,w)$, we have
			\begin{equation}
			\norm{\psi_{R,0}^{\delta}f}_{L^p(I,w)}\leq C\norm{f}_{L^p(I,w)}.
			\end{equation}	
			where $C$ is a constant independent of $f$ and $R.$
		\end{proposition}
		\begin{proof}
			It suffices to use the same techniques as those used in the previous proof to get an estimate of $\norm{\psi_{R,0}^{\delta}f}_{L^p(I_1,w)}$ and $\norm{\psi_{R,0}^{\delta}f}_{L^p(I_2,w)}$ for all $f\in L^p(I,w),$ where $I=(a,b)=I_1\cup I_2$ with $I_1=(x_0-r^{\alpha}_0,x_0+r^{\alpha}_0)$ and $I_2=\{y, |y-x_0|>r^{\alpha}_0\}$ where $r^{\alpha}_0=\frac{r}{R^{\mu(p)}}.$ 
		\end{proof}
		To conclude the theorem's proof it suffices to find a uniform bound of  $\mathcal{R}_{R}^{\delta}$.
		\begin{proposition}
			Let $1\leq p\leq 2$ and $\delta>\delta(p)=\frac{\gamma(p')}{\varepsilon}$, then for all $f\in L^p(I,w)$, we have
			\begin{equation}
			\norm{\mathcal{R}_{R}^{\delta}f}_{L^p(I,w)}\leq C\norm{f}_{L^p(I,w)}.
			\end{equation}
			where $C$  depends only on $p$. 
		\end{proposition}
		\begin{proof}
			From Holder's inequality and the previous lemma, we have
			\begin{eqnarray*}
				\norm{\mathcal{R}_{R}^{\delta}f}^2_{L^p(I,w)}&\leq&2^{2(\frac{1}{p}-\frac{1}{2})}\norm{\mathcal{R}_{R}^{\delta}f}^2_{L^2(I,w)}\\&\leq&2^{2(\frac{1}{p}-\frac{1}{2})}\sum_{k=K_R+1}^{\infty}\sum_{n=0}^{\infty}\norm{\phi_{R,k}^{\delta}(\lambda_n)a_n(f)\varphi_n}_{L^2(I,w)}\\&\leq&C2^{2(\frac{1}{p}-\frac{1}{2})}\sum_{k=K_R+1}^{\infty}2^{-2k\delta}\sum_{R_{k,1}\leq\lambda_n\leq R_{k,2}}\norm{a_n(f)\varphi_n}^2_{L^2(I,w)}\\&\leq&C2^{2(\frac{1}{p}-\frac{1}{2})}\sum_{k=K_R+1}^{\infty}2^{-2k(\delta+\frac{1}{2})}R^{2(\frac{1}{2}+\frac{\gamma(p')}{\varepsilon})}\norm{f}^2_{L^p(I,\omega)}\\&\leq&C2^{2(\frac{1}{p}-\frac{1}{2})}2^{-2(\delta+\frac{1}{2})\big(\left[\frac{\log(R)}{\log(2)}\right]+1\big)}R^{2(\frac{1}{2}+\frac{\gamma(p')}{\varepsilon})}\norm{f}^2_{L^p(I,\omega)}\\&\leq&C2^{2(\frac{1}{p}-\frac{1}{2})}R^{-2(\delta-(\frac{\gamma(p')}{\varepsilon}))}\norm{f}^2_{L^p(I,\omega)}
			\end{eqnarray*}
			Finally we obtain
			\begin{eqnarray*}
				\norm{\mathcal{R}_{R}^{\delta}f}_{L^p(I,w_{\alpha,\beta})}&\leq& C2^{(\frac{1}{p}-\frac{1}{2})}R^{-(\delta-(\frac{\gamma(p')}{\varepsilon}))}\norm{f}_{L^p(I,\omega_{\alpha,\beta})}\\&\leq&C(p)\norm{f}_{L^p(I,\omega_{\alpha,\beta})}.
			\end{eqnarray*}
		\end{proof}
	\end{proof}
	\begin{corollary}
		Under the notation and conditions of the previous Theorem, we have for all $f\in L^p(I,w)$
		\begin{equation}
		\Psi_R^{\delta}f \to f ~~\mbox{as}~~ R\to \infty.
		\end{equation}
	\end{corollary}
	\begin{proof}
		{\bf Step1: } We prove that, for every $ f \in \mathcal{C}^{\infty}(I,\R) $, $\Psi^\delta_Rf \to f $ in $L^p(I,\omega)$. Note that
		\begin{eqnarray}
		\Big| \Big(1-\frac{\lambda_n}{R}\Big)^\delta_+ \scal{f,\varphi_n}_{L^2(I,\omega)} \Big| &\leq & \Big| \scal{f,\varphi_n}_{L^2(I,\omega)} \Big| = \frac{1}{\lambda_n}\Big|\scal{f,\mathcal{L}\varphi_n}_{L^2(I,\omega)} \Big| \nonumber\\
		&=&\frac{1}{\lambda_n}\Big|\scal{\mathcal{L}f,\varphi_n}_{L^2(I,\omega)} \Big| = \cdots = \frac{1}{\lambda_n^k}\Big|\scal{\mathcal{L}^k.f,\varphi_n}_{L^2(I,\omega)} \Big| \nonumber \\
		&\leq & n^{-k\varepsilon} \norm{\mathcal{L}^k.f}_{L^2(I,\omega)}.
		\end{eqnarray}
		Since $ \norm{\varphi_n}_{L^2(I,\omega)} \leq n^{\gamma(p)} $, it suffices to take $k$ big enough to have $ \gamma(p)-k\varepsilon < -1 $ and obtain the convergence of the series in $ L^p(I,\omega) $.\\
		Since $ \displaystyle \norm{\Psi^\delta_R.f-f}_2^2 = \sum_{n=0}^\infty \Big((1-\frac{\lambda_n}{R})^\delta_+ - 1\Big)^2 |a_n(f)|^2 \to 0 $ as $R \to \infty$,then the result remains true for $1\leq p<\infty$. \\
		{\bf Step2: }For all $ \varepsilon >0 $. By density of $\mathcal{C}^\infty_0(I,\R) $ in $L^p(I,\omega)$,  there exists $g \in \mathcal{C}^\infty_0(I,\R) $ such that $ \norm{f-g}_{L^p(I,\omega)} < \varepsilon $ and there exists $R>0$ such that $\norm{\Psi^\delta_R. f-\Psi^\delta_R.g}_{L^p(I,\omega)} < \varepsilon $.\\
		By writing, 
		$$ \norm{\Psi^\delta_R.f-f}_{L^p(I,\omega)} \leq \norm{\Psi^\delta_R. f-\Psi^\delta_R.g}_{L^p(I,\omega)} + \norm{\Psi^\delta_R. g-g}_{L^p(I,\omega)} + \norm{f-g}_{L^p(I,\omega)}, $$
		one gets the desired result.
	\end{proof}
To conclude for the proof of sufficient conditions of both theorems 1 and 2, it suffices to verify that the two considered bases satisfy conditions (A) and (B).We will prove this result only for the case of GPSWFs. The other case is almost identical.\\ 
We first recall that from \eqref{diff_oper1}, the GPSWFs are the eigenfunctions of the
Sturm-Liouville operator $\mathcal L_c^{(\alpha)}.$ 
Also, note that  the $(n+1)-$th eigenvalue $\chi_{n,\alpha}(c)$ of $\mathcal L_c^{(\alpha)}$ satisfies the following classical inequalities,
\begin{equation*}
n^2\leq n(n+2\alpha+1) \leq \chi_{n,\alpha}(c) \leq n (n+2\alpha+1) +c^2,\quad \forall n\geq 0.
\end{equation*}
Moreover, for every $0\leq m<M$ such that $M-m>1$, we have
\begin{eqnarray*}
	\sum_{\chi_{n,\alpha}(c)\in (m,M)}1&\leq&\sum_{n(n+2\alpha+1)\in(\max(0,m-c^2),M)}1\\&\leq&\sum_{(n+\alpha+1/2)^2-(\alpha+1/2)^2\in(\max(0,m-c^2),M)}1\\&\leq&\sum_{n\in \left((\max(0,m-c^2)+(\alpha+1/2)^2)^{\frac{1}{2}}-(1/2+\alpha),(M+(\alpha+1/2)^2)^{\frac{1}{2}}-(1/2+\alpha)\right)}1\\&\leq& C(M-m).
\end{eqnarray*} 
It follows that condition (B) is satisfied.\\
Form \cite{Boulsane-Jaming-Souabni} Lemma $2.6$, one can conclude that condition (A) is satisfied for weighted prolate spheroidal wave functions for $1<p<\infty $. 
Moreover, it has been shown in \cite{Karoui-Souabni1} that 
$ \norm{\psi^{(\alpha)}_{n,c}}_{\infty} \leq C \Big(\chi_{n,\alpha}(c)\Big)^{\frac{\alpha+1}{2}}.$
Then, by using \eqref{boundschi}, we obtain $\norm{\psi^{(\alpha)}_{n,c}}_{1} \leq C \Big(\chi_{n,\alpha}(c)\Big)^{\frac{\alpha+1}{2}}\leq C n^{\alpha+1}.$\\
 
\begin{remark}
	The uniform norm of the CPSWFs has been given in \cite{Karoui-Boulsane}.
\end{remark}

	\section{Proof of necessary condition}
	The transferring theorem  from the uniform boundedness of $\Psi_R^{\delta}$ to the uniform boundedness of the Hankel multiplier transform operator $\mathcal{M}_{\alpha}$ defined by $\mathcal{M}_{\alpha}(f)=\mathcal{H}_{\alpha}\left(\phi(.)\mathcal{H}_{\alpha}(f)\right)$ can be used to derive necessary condition.
	Note here that $\phi$ is a bounded function on $\R$, continuous except on a set of Lebesgue measure zero  and $\mathcal{H}_{\alpha}$ is the modified Hankel operator defined by $$\mathcal{H}_{\alpha}(f)(x)=\int_0^{\infty}\frac{J_{\alpha}(xy)}{(xy)^{\alpha}}f(y)y^{2\alpha+1}dy.$$
	From \cite{C.V} and the transferring theorem, the uniform boundedness of $\Psi_{R}^{\delta}$ holds true if and only if $\delta>\max\{2(\alpha+1)|\frac{1}{p}-\frac{1}{2}|-\frac{1}{2}, 0\}.$ It's easy to check that $\max\{2(\alpha+1)|\frac{1}{p}-\frac{1}{2}|-\frac{1}{2}, 0\}\geq \max\{\frac{\gamma_{\alpha}(p')}{2},0\}$ for every $p\not=2-\frac{1}{\alpha+3/2}$, then one gets our necessary condition.
	 To be more precise, let's study each transferring theorem separately.
	\subsection{GPSWFs's case}
	 Let's recall that the family of weighted prolate spheroidal wave functions $\{\ps(\cos \theta)\}_n$ form an orthonormal system on $(0,\pi)$ with respect to the measure $ (\sin\theta)^{2\alpha+1}  d\theta $. \\
	For a function $f(\theta) $ integrable on $(0,\pi)$ with respect to the measure defined above, we have formally 
	$$ f(\theta) = \sum_{n=0}^{\infty} a_n(f) \ps(\cos\theta) \qquad a_n(f)= \int_{0}^{\pi} f(\theta)\ps(\cos \theta) (\sin\theta)^{2\alpha+1}d\theta $$
	For $p\geq1$ and a function $f$ on $(0,\pi)$ we define a norm 
	$$ \norm{f}_p = \Bigg( \int_{0}^{\pi} |f(\theta)|^p (\sin\theta)^{2\alpha+1} \Bigg)^{1/p}. $$
	
	Before stating an adequate transferring theorem, let's define a GPSWFs-multiplier.
	
	\begin{definition}
		Let $\lambda >0$ be a sufficiently large real, the bounded sequence $\{ \phi(\frac{\chi^{1/2}_{n,\alpha}(c)}{\lambda}) \}_n $ is called a Weighted prolate multiplier if there exist a constant $C>0$ such that for every $f \in L^p(I,\omega_\alpha)$, we have 
		$$ \norm{\sum_{n=0}^{\infty}\phi(\frac{\chi^{1/2}_{n,\alpha}(c)}{\lambda}) a_n(f) \ps}_{L^p(I,\omega_\alpha)} \leq C \norm{f}_{L^p(I,\omega_\alpha)}. $$
		The smallest constant $C$ verifying this last inequality is written $\norm{\phi(\frac{\chi^{1/2}_{n,\alpha}(c)}{\lambda})}_p$. In the same context, the function $\phi$ is called an $\L^p$-Hankel transform multiplier if $\mathcal{M}_\alpha(f) = \h_\alpha(\phi(.)\h_\alpha(f)) $ is uniformly bounded on $L^p\left((0,\infty),\theta^{2\alpha+1} d\theta\right)$.
	\end{definition}
	
		\begin{theorem}[Transferring theorem]
		Let $1<p<\infty$, $0\leq\alpha <3/2$ and $\phi$ be a bounded function on $(0,\infty)$ continuous except on a set of Lebesgue measure zero such that $\{\phi(\frac{\chi^{1/2}_{n,\alpha}(c)}{\lambda})\}_n$ is a Weighted prolate multiplier for all large $\lambda>0$ and $\displaystyle \liminf_{\lambda\to \infty} \norm{\phi(\frac{\chi^{1/2}_{n,\alpha}(c)}{\lambda})}_p$ is finite then $\phi$ is an $L^p$-Hankel transform multiplier and we have 
		$$ \norm{\mathcal{M}_\alpha}_p \leq \liminf_{\lambda\to \infty}\norm{\phi\left(\frac{\chi^{1/2}_{n,\alpha}(c)}{\lambda}\right)}_p.$$
	\end{theorem}
	
	\begin{proof}
		Let $g$ be an infinitely differentiable function with compact support in $[0,M]$ and put $g_\lambda(\theta)=g(\lambda \theta) $. Here $\lambda$ is a positive real so that $supp(g_\lambda) \subset [0,\pi]$.\\
		Recall that we have by assumption 
		\begin{equation}\label{hyp}
		\norm{\sum_{n=0}^{\infty}\phi(\chi^{1/2}_{n,\alpha}(c)/\lambda)a_n(g) \ps(\cos(.))}_p \leq \norm{\phi(\chi^{1/2}_{n,\alpha}(c)/\lambda)}_p \norm{g}_p.
		\end{equation}
		Via a simple change of variable, one can write
		$$ \lim_{\lambda \to \infty}\lambda^{2\alpha+2} \norm{g_\lambda}_p^p = \lim_{\lambda \to \infty} \int_{0}^{M} |g(\tau)|^p\Big(\lambda \sin(\tau/\lambda)\Big)^{2\alpha+1} d\tau = \int_{0}^{\infty} |g(\tau)|^p \tau^{2\alpha+1} d\tau .$$
		By using \eqref{hyp} together with Fatou's lemma, one gets
		\begin{eqnarray}
		& \displaystyle\int_{0}^{\infty}&\liminf_{\lambda\to \infty} \Big|\chi_{(0,\pi\lambda)}(\tau)\sum_{n=0}^{\infty} \phi(\chi^{1/2}_{n,\alpha}(c)/\lambda)a_n(g_\lambda) \ps(\cos\tau/\lambda)\Big|^p \tau^{2\alpha+1} d\tau \nonumber\\ 
		&=& \int_{0}^{\infty} \liminf_{\lambda\to \infty} \Big|\chi_{(0,\pi\lambda)}(\tau)\sum_{n=0}^{\infty} \phi(\chi^{1/2}_{n,\alpha}(c)/\lambda)a_n(g_\lambda) \ps(\cos\tau/\lambda)\Big|^p \lambda^{2\alpha+1}\sin(\tau/\lambda)^{2\alpha+1} d\tau \nonumber \\
		&\leq& \liminf_{\lambda\to \infty}\lambda^{2\alpha+1} \int_{0}^{\infty} \Big|\chi_{(0,\pi\lambda)}(\tau)\sum_{n=0}^{\infty} \phi(\chi^{1/2}_{n,\alpha}(c)/\lambda)a_n(g_\lambda) \ps(\cos\tau/\lambda)\Big|^p \sin(\tau/\lambda)^{2\alpha+1} d\tau \nonumber \\
		&\leq& \liminf_{\lambda\to \infty}\lambda^{2\alpha+2}\norm{\phi(\chi^{1/2}_{n,\alpha}(c)/\lambda)}_p \norm{g_\lambda}^p_p = \liminf_{\lambda\to \infty}\norm{\phi(\chi^{1/2}_{n,\alpha}(c)/\lambda)}_p \Bigg[\int_{0}^{\infty}|g(\tau)|^p \tau^{2\alpha+1} d\tau\Bigg]. \nonumber
		\end{eqnarray}
		Then there exists a sequence $\lambda_1<\lambda_2<\cdots<\lambda_p \to \infty$ that $G(\tau,\lambda) = \displaystyle \chi_{(0,\pi\lambda)}(\tau)\sum_{n=0}^{\infty} \phi(\chi^{1/2}_{n,\alpha}(c)/\lambda)a_n(g_\lambda) \ps(\cos\frac{\tau}{\lambda})$ converges weakly to a function $G(\tau)$. Furthermore, $G$ satisfies 
		$$ \Bigg[\int_{0}^{\infty}|G(\tau)|^p \tau^{2\alpha+1} d\tau\Bigg]^{1/p} \leq \liminf_{\lambda\to \infty}\norm{\phi(\chi^{1/2}_{n,\alpha}(c)/\lambda)}_p \Bigg[\int_{0}^{\infty}|g(\tau)|^p \tau^{2\alpha+1} d\tau\Bigg]^{1/p}. $$
		Let us now prove that $G = \h_{\alpha}(\phi. \h_{\alpha}(g))$. \\
		Let $$G(\tau,\lambda) = \chi_{(0,\pi\lambda)}(\tau)\Big[\sum_{n=0}^{N[\lambda]}+\sum_{N[\lambda]+1}^{\infty}\Big] \phi(\chi^{1/2}_{n,\alpha}(c)/\lambda)a_n(g_{\lambda}) \ps(\cos\tau/\lambda) = G^N(\tau,\lambda) + H^N(\tau,\lambda)$$
		We start by giving the following lemma that will be proved later,
		\begin{lemma}\label{rest}
			We have $$ \int_{0}^{\infty} \Big|H^N(\tau,\lambda)\Big|^2 \tau^{2\alpha+1} d\tau = O(\frac{1}{N}) \mbox{ uniformly in }\lambda $$
		\end{lemma}
		Therefore, by the diagonal argument, there exists a subsequence also noted $\{\lambda_j\}$ for a sake of clarity, such that $H^N(\tau,\lambda_j)$ converges weakly to a function $H^N(\tau)$ and $$\displaystyle \int_{0}^{\pi} \Big|H^N(\tau)\Big|^p \tau^{2\alpha+1} d\tau = O(\frac{1}{N^2}).$$\\
		Then, there exists a subsequence $H^{N_j}$ denoted for the same reason $H^{N}$ that converges to zero a.e . \\
		Since $ G^N(\tau,\lambda) = G(\tau,\lambda) - H^N(\tau,\lambda) $, $G^N(\tau,\lambda)$ converge weakly to a limit $G^N(\tau)$ and $G(\tau) = G^N(\tau) + H^N(\tau)$. Thus $G^N(\tau)$ converges to $G(\tau)$ almost everywhere. On the other hand, we will prove the following lemma 
		\begin{lemma}\label{part1}
			We have $$ \lim_{\lambda \to \infty} G^N(\tau,\lambda) = \int_{0}^{N} \phi(v) \h_\alpha.g(v) \frac{J_\alpha(v\tau)}{(v\tau)^\alpha}v^{2\alpha+1} dv, $$
		\end{lemma}
		which implies that  
		$$ G(\tau) =\int_{0}^{\infty} \phi(v) \h_\alpha.g(v) \frac{J_\alpha(v\tau)}{(v\tau)^\alpha}v^{2\alpha+1} dv, $$ and achieves our proof.
	\end{proof}
	
	\begin{proof}[Proof of Lemma 2]
		We have 
		\begin{eqnarray}\label{1}
		\int_{0}^{M}|H^N(\tau,\lambda)|^2 \Big(\lambda \sin\frac{\tau}{\lambda}\Big)^{2\alpha+1} d\tau &=& \lambda^{2\alpha+2}\int_{0}^{\pi} |H^N(\lambda\tau,\lambda)|^2(\sin\tau)^{2\alpha+1}d\tau \nonumber \\ 
		&=& \lambda^{2\alpha+2} \sum_{N[\lambda]+1}^{\infty} |\phi(\frac{n}{\lambda})|^2 |a_n(g_\lambda)|^2.
		\end{eqnarray}
		
		Recall that in \cite{Karoui-Souabni2}, authors have given the following uniform approximation of GPSWFs in term of Jacobi polynomials for $0\leq\alpha<3/2$,
		\begin{equation}\label{approx-jacobi}
		\ps(\cos \theta) = A_n \wJ_n^{(\alpha)}(\cos\theta)+R_{n,c}^{(\alpha)}(\cos\theta) \qquad 
		\norm{R_{n,c}}^{(\alpha)}_\infty \leq C_{\alpha,c} \frac{1}{2n+2\alpha+1}.
		\end{equation}
		We also know that (see for example \cite{Szego})
		\begin{equation}\label{jacobi}
		n (\sin\theta)^{2\alpha+1} \wJ_n^{(\alpha,\alpha)}(\cos\theta) = 2 \frac{h^{\alpha+1}_{n-1}}{h^{(\alpha)}_{n}} \frac{d}{d\theta} \Big[(\sin\theta)^{2\alpha+2} \wJ_{n-1}^{(\alpha+1,\alpha+1)}(\cos\theta)\Big]
		\end{equation}
		By combining \eqref{approx-jacobi} and \eqref{jacobi}, one gets
		$$ (\sin\theta)^{2\alpha+1} \ps(\cos\theta) = \frac{2}{n} \frac{h^{\alpha+1}_{n-1}}{h^{(\alpha)}_{n}} \frac{d}{d\theta} \Big[(\sin\theta)^{2\alpha+2} \wJ_{n-1}^{(\alpha+1,\alpha+1)}(\cos\theta)\Big] + R_{n,c}^{(\alpha)} (\cos\theta).$$
		Then, integrating by parts one gets
		\begin{eqnarray}
		a_n(g_\lambda) &=& \frac{C}{n} \int_{0}^{\pi}\frac{g'(\lambda\theta)}{\sin\theta} \wJ_{n-1}^{\alpha+1}(\cos\theta) (\sin \theta)^{2\alpha+3} d\theta + \int_{0}^{\pi} R_{n,c}^{(\alpha)}(\cos \theta) g(\lambda \theta) d\theta \nonumber \\
		&=& a_{n,1}(g_\lambda) + a_{n,2}(g_\lambda) \nonumber
		\end{eqnarray}
		Let's come back to \eqref{1}. We have by Bessel's inequality
		\begin{eqnarray}
		\lambda^{2\alpha+2} \sum_{N[\lambda]+1}^{\infty} |\phi(\frac{n}{\lambda})|^2 |a_{n,1}(g_\lambda)|^2 &\leq& C \lambda^{2\alpha+2} \Big[\frac{\lambda}{N(\lambda-1)}\Big]^2 \sum_{N[\lambda]+1}^{\infty} |\frac{n}{\lambda} a_{n,1}(g_\lambda)|^2 \nonumber \\
		&\leq& \frac{C}{N^2}\lambda^{2\alpha+2} \int_{0}^{\pi} \Big| \frac{g'(\lambda \theta)}{\sin \theta} \Big|^2 (\sin \theta)^{2\alpha+3} d\theta \nonumber \\
		&=& \frac{C}{N^2} \int_{0}^{M} |g'(\theta)|^2 \Big(\lambda\sin\frac{\theta}{\lambda}\Big)^{2\alpha+1} d\theta \nonumber \\
		&=& O(\frac{1}{N^2}) \mbox{  uniformly in } \lambda
		.
		\end{eqnarray} 
		On the other hand, using Cauchy-Schwarz's inequality
		\begin{eqnarray}
		\lambda^{2\alpha+2} \sum_{N[\lambda]+1}^{\infty} |\phi(\frac{n}{\lambda})|^2 |a_{n,2}(g_\lambda)|^2 &\leq& 
		C \lambda^{2\alpha+2} \sum_{N[\lambda]+1}^{\infty} \norm{R_{n,c}^{\alpha}}^2_2 \norm{g(\lambda . )}_2^2 \nonumber \\
		&\leq& C \sum_{N[\lambda]+1}^{\infty} \frac{1}{n^2} \int_{0}^{M}|g(\theta)|^2 \big(\lambda \sin\frac{\theta}{\lambda}\big)^{2\alpha+1} d\theta \nonumber \\
		&=& O(\frac{1}{N}). \nonumber
		\end{eqnarray}
		Then, one conclude that $$ \int_{0}^{M}|H^N(\tau,\lambda)|^2 \tau^{2\alpha+1} d\tau = O(\frac{1}{N}) \mbox{ uniformly in }\lambda. $$
		
	\end{proof}
	
	\begin{proof}[Proof of lemma 3]
		We use now the following uniform approximation of GPSWFs in term of Bessel function (we refer the reader once again to \cite{Karoui-Souabni2})
		\begin{equation}\label{approx-Jacobi}
		\ps(\cos\frac{\tau}{\lambda}) = A_\alpha(q) \frac{\chi_{n,c}^{1/4}S(\cos\frac{\tau}{\lambda})^{1/2} J_\alpha(\chi_{n,c}^{1/2} S(\cos\frac{\tau}{\lambda})) }{(\sin\frac{\tau}{\lambda})^{\alpha+1/2}(1-q\cos^2\frac{\tau}{\lambda})^{1/4} } + E_{n,c}(\cos\frac{\tau}{\lambda}) ,
		\end{equation}
		where $$ \Big|E_{n,c} (\cos\theta)\Big| \leq \frac{C.A_\alpha(q)}{(1-q)} \frac{(\sin\theta)^{1/2}}{(1-q\cos^2\theta)^{1/4}} \qquad \forall \theta\in [0,\pi]  \quad \mbox{and} \quad S(x) = \int_{x}^{1} \sqrt{\frac{1-qt^2}{1-t^2}}dt. $$
		
		Note that it has also been shown in \cite{Karoui-Bonami} that 
		$$ \frac{\sin\theta \sqrt{1-q\cos^2\theta}}{S(\cos\theta)} = 1 + \Big(\frac{q}{1-q}+\frac{3}{4}\Big)(1-\cos\theta) + o(1-\cos\theta) .$$
		Thus, we can write, for $n\leq N[\lambda]$,and by taking into account that $\sqrt{x}J_\alpha(x)$ is bounded then
		\begin{eqnarray}
		\frac{\ps(\cos\frac{\tau}{\lambda})}{\lambda^\alpha} &=& n^{1/2}\frac{J_{\alpha}(\frac{n\tau}{\lambda})}{\Big(\lambda.\sin\frac{\tau}{\lambda}\Big)^\alpha} -n^{1/2}\frac{J_{\alpha}(\frac{n\tau}{\lambda})}{\Big(\lambda.\sin\frac{\tau}{\lambda}\Big)^\alpha}\big(\frac{q}{1-q}+3/4\big)\frac{\tau^2}{4\lambda^{\alpha+2}}+ O(\frac{1}{n.\lambda^{\alpha+2}}) \nonumber \\
		&=& n^{1/2} J_{\alpha}(\frac{n\tau}{\lambda}) \Big(\frac{1}{\tau}\Big)^\alpha + o(\frac{1}{n}).
		\end{eqnarray}
		On the other hand,
		\begin{eqnarray}
		\lambda^\alpha a_n(g_\lambda) &=& \lambda^{\alpha-1} \int_{0}^{M}g(\tau) \ps(\cos\frac{\tau}{\lambda}) (\sin\frac{\tau}{\lambda})^{2\alpha+1} d\tau \nonumber\\
		&=& \frac{1}{\lambda^2} \Bigg[ A_\alpha(q) n^{1/2} \int_{0}^{\infty} g(\tau) J_\alpha(\frac{n\tau}{\lambda}) \Big(\lambda \sin\frac{\tau}{\lambda}\Big)^{\alpha+1} d\tau\Bigg] + o(\frac{1}{\lambda^2}) \nonumber\\
		&=& \frac{n^{1/2}}{\lambda^2} \int_{0}^{\infty} g(\tau)J_\alpha(\frac{n\tau}{\lambda}) \tau^{\alpha+1} d\tau + o(\frac{1}{\lambda^2}). \nonumber
		\end{eqnarray}
		Then, by combining the last two estimates, one gets
		$$ G^N(\tau,\lambda) = \sum_{n=0}^{N[\lambda]+1} \phi(\frac{n}{\lambda})\h_{\alpha}.g(\frac{n}{\lambda}) J_\alpha(\frac{n\tau}{\lambda}) \frac{1}{\tau^\alpha} \frac{n}{\lambda^2} +\frac{n}{\lambda^2} o(1).$$
		Therefore, by letting $\lambda \to \infty$, we conclude for the proof of lemma\ref{part1}.
		
	\end{proof}

	\subsection{CPSWF's case}
	
As for the example studied in the previous section, we start by establishing an adequate transferring theorem for the circular case. To do this, we introduce a suitable terminology.
\begin{definition}
	Let $\lambda>0$ be a sufficiently large real, a bounded sequence $\{m(\frac{\chi^{1/2}_{n,\alpha}(c)}{\lambda})\}_n$ is called to be a Circular prolate multiplier, if there exists a constant $C>0$ such that for every $f \in L^p(0,1)$, we have 
	$$ \norm{\sum_{n=0}^{\infty}m(\frac{\chi^{1/2}_{n,\alpha}(c)}{\lambda}) a_n(f) \varphi^{(\alpha)}_{n,c}}_{L^p(0,1)} \leq C \norm{f}_{L^p(0,1)}. $$
	The smallest constant $C$ verifying the last inequality is written $ \norm{m\Big(\frac{\chi^{1/2}_{n,\alpha}(c)}{\lambda}\Big)}_p$.  
\end{definition}

Here $ \mathcal{M}:= \mathcal{M}_{0} = \mathcal{H}_0\Big(m(.)\mathcal{H}_0(f)\Big)$ is the multiplier related to the Hankel transform operator.

\begin{theorem}[Circular transferring theorem]
	Let $1<p<\infty$, $\alpha\geq 1/2$ and $m$ be a bounded function on $(0,\infty)$ continuous except on a set of Lebesgue measure zero such that $\{m(\frac{\chi^{1/2}_{n,\alpha}(c)}{\lambda})\}_n$ is a Circular prolate multiplier for all large $\lambda>0$ and $\displaystyle \liminf_{\lambda\to \infty} \norm{m(\frac{\chi^{1/2}_{n,\alpha}(c)}{\lambda})}_p$ is finite then $m$ is an $L^p$-Hankel transform multiplier  and we have 
	$$ \norm{\mathcal{M}}_p \leq \liminf_{\lambda\to \infty}\norm{m\left(\frac{\chi^{1/2}_{n,\alpha}(c)}{\lambda}\right)}_p.$$
\end{theorem}

\begin{proof}
Let $\lambda>0$ and $g\in C^{\infty}_c(0,\infty)$  supported in $(0,M)$ such that $\lambda> \frac{2}{\pi}M$. Let $g_{\lambda}(\tau)=g(\lambda \tau)$ for every $\tau\in (0,1)$ and $G_{\lambda}=g_{\lambda}\circ\arccos.$

By asymption, we have

\begin{equation*}
\norm{\sum_{n=0}^{\infty}m\left(\frac{\chi^{1/2}_{n,\alpha}(c)}{\lambda}\right) a_n(G_{\lambda}) \varphi_n}_{L^p\left(0,1\right)}\leq\norm{m(\frac{\chi^{1/2}_{n,\alpha}(c)}{\lambda})}_p\norm{G_{\lambda}}_{L^p\left(0,1\right)}.
\end{equation*}

Then, we get
\begin{equation*}
\norm{\chi_{(0,\lambda\frac{\pi}{2})}\sum_{n=0}^{\infty}m\left(\frac{\chi^{1/2}_{n,\alpha}(c)}{\lambda}\right) a_n\left(G_{\lambda}\right) \varphi_n(\cos(\frac{.}{\lambda}))}^p_{L^p((0,\infty),\sin(\frac{.}{\lambda}))}\leq\norm{m(\frac{\chi^{1/2}_{n,\alpha}(c)}{\lambda})}^p_p\norm{g}^p_{L^p\left((0,\infty),\sin(\frac{.}{\lambda})\right)}
\end{equation*}
We denote by $$F_{\lambda}(\theta)=\chi_{(0,\lambda\frac{\pi}{2})}(\theta)\sum_{n=0}^{\infty}m\left(\frac{\chi^{1/2}_{n,\alpha}(c)}{\lambda}\right) a_n\left(G_{\lambda}\right) \varphi_n\left(\cos(\frac{\theta}{\lambda})\right),$$
hence we have
\begin{equation}\label{e1}
\norm{F_{\lambda}}^p_{L^p((0,\infty),\sin(\frac{.}{\lambda}))}\leq\norm{m(\frac{\chi^{1/2}_{n,\alpha}(c)}{\lambda})}^p_p\norm{g}^p_{L^p((0,\infty),\sin(\frac{.}{\lambda}))}.
\end{equation}
By using \eqref{e1}, Fatou's Lemma and the fact that $\displaystyle\lim_{\lambda\to\infty}\lambda\sin(\frac{\theta}{\lambda})=\theta$, we obtain
\begin{equation}
\norm{\displaystyle\liminf_{\lambda\to\infty}F_{\lambda}}^p_{L^p((0,\infty),\theta d\theta)}\leq\displaystyle\liminf_{\lambda\to\infty}\norm{m(\frac{\chi^{1/2}_{n,\alpha}(c)}{\lambda})}^p_p\norm{g}^p_{L^p((0,\infty),\theta d\theta)}
\end{equation}
Let $L=\displaystyle\liminf_{\lambda\to\infty}\norm{m(\frac{\chi^{1/2}_{n,\alpha}(c)}{\lambda})}_p$, then there exists a sequence of $(\lambda_j)_{j\in\mathbb{N}}$ such that $\displaystyle\lim_{j\to\infty}\lambda_j=+\infty$ verifying 
\begin{equation}\label{eq8}
\norm{F_{\lambda_j}}_{L^p((0,\infty),\theta d\theta)}\leq (L+1/j)\norm{g}_{L^p((0,\infty),\theta d\theta)}.
\end{equation}
On the other hand, as $m$ is bounded and from Perseval's formula, we have
\begin{equation}\label{eq9}
\norm{F_{\lambda_j}}_{L^2((0,\infty),\theta d\theta)}\leq (L+1/j)\norm{g}_{L^2((0,\infty),\theta d\theta)}.
\end{equation}
From \eqref{eq8} and \eqref{eq9} there exists a subsequence of $(\lambda_j)_{j\in\mathbb{N}}$ denoted also $(\lambda_j)_{j\in\mathbb{N}}$ such that the sequence $\{F_{\lambda_j}\}$ converge weakly to a function $F$ in $L^p\cap L^2((0,\infty),\theta d\theta)$ and satisfying the following inequality
\begin{equation}\label{eq10}
\norm{F}_{L^p((0,\infty),\theta d\theta)}\leq L\norm{g}_{L^p((0,\infty),\theta d\theta)}.
\end{equation}
Our purpose now is to show that $F=\mathcal{H}_0\left(m(.)\mathcal{H}_0(g)\right)$ almost everywhere on $(0,\infty).$\\
Let $N\geq 1$ and $\theta\in (0,\infty)$
\begin{eqnarray*}
	F_{\lambda}(\theta)&=&\chi_{(0,\lambda)}(\theta)\sum_{n=0}^{\infty}m\left(\frac{\chi^{1/2}_{n,\alpha}(c)}{\lambda}\right) a_n(G_{\lambda}) \varphi_n\left(\cos(\frac{\theta}{\lambda})\right)\\&=&\chi_{(0,\lambda)}(\theta)\left[\sum_{n=0}^{N[\lambda]}+\sum_{n=N[\lambda]+1}^{\infty}\right]m\left(\frac{\chi^{1/2}_{n,\alpha}(c)}{\lambda}\right) a_n(G_{\lambda})\varphi_n\left(\cos(\frac{\theta}{\lambda})\right)\\&=&F^N_{\lambda}(\theta)+K^N_{\lambda}(\theta).
\end{eqnarray*}
Using \eqref{differ_operator2}, the function $F(\theta)=\varphi^{\alpha}_{n,c}\left(\cos(\theta)\right)$ satisfies the following differential equation
\begin{eqnarray*}
	\mathcal{L}(F)(\theta)&=&-F''(\theta)-\frac{\cos(\theta)}{\sin(\theta)}F'(\theta)+\left(c^2\cos^2(\theta)-\frac{1/4-\alpha^2}{\cos^2(\theta)}\right)F(\theta)\\&=&\chi_{n,\alpha}(c)F(\theta).
\end{eqnarray*}
Using the symmetry of $\mathcal{L}$ on $C^{\infty}_c(0,\infty)$, we obtain
\begin{eqnarray*}
	a_n(G_{\lambda})&=&\scal{G_{\lambda},\varphi^{\alpha}_{n,c}}_{L^2(0,1)}\\&=&\frac{1}{\chi_{n,\alpha}(c)}\int_0^{\frac{\pi}{2}}g_{\lambda}(\theta)\chi_{n,\alpha}(c)\varphi^{\alpha}_{n,c}\left(\cos(\theta)\right)\sin(\theta)d\theta\\&=&\frac{1}{\chi_{n,\alpha}(c)}\int_0^{\frac{\pi}{2}}g_{\lambda}(\theta)\mathcal{L}(F)(\theta)\sin(\theta)d\theta\\&=&\frac{\lambda^2}{\chi_{n,\alpha}(c)}\int_0^{\frac{\pi}{2}}\frac{1}{\lambda^2}\mathcal{L}\left(g_{\lambda}\right)(\theta)F(\theta)\sin(\theta)d\theta=\frac{\lambda^2}{\chi_{n,\alpha}(c)}a_n\left(\frac{1}{\lambda^2}\mathcal{L}\left(g_{\lambda}\right)\right).
\end{eqnarray*}

 Using the previous equality, Perseval's formula and the fact that $m$ is bounded with the well known inequality $\frac{2}{\pi}\theta\leq \sin(\theta)\leq\theta$, for $0\leq \theta\leq \frac{\pi}{2}$ and \eqref{boundschi2}, we obtain
\begin{eqnarray*}
	\norm{K_{\lambda}^{N}}_{L^2((0,\infty),\theta d\theta)}&=&\left[\int_0^{\infty}\chi_{(0,\lambda\frac{\pi}{2})}(\theta)\left|\sum_{n=N[\lambda]+1}^{\infty}m\left(\frac{\chi^{1/2}_{n,\alpha}(c)}{\lambda}\right)  a_n(G_{\lambda})\varphi_n(\cos(\frac{\theta}{\lambda}))\right|^2 \theta d\theta\right]^{1/2}\\&=&\left[\int_0^{\lambda\frac{\pi}{2}}\left|\sum_{n=N[\lambda]+1}^{\infty}m\left(\frac{\chi^{1/2}_{n,\alpha}(c)}{\lambda}\right)  a_n(G_{\lambda})\varphi_n(\cos\left(\frac{\theta}{\lambda})\right)\right|^2 \theta d\theta\right]^{1/2}\\&\leq&\sqrt{\frac{\pi}{2}}\left[\lambda\int_0^{\lambda\frac{\pi}{2}}\left|\sum_{n=N[\lambda]+1}^{\infty}m\left(\frac{\chi^{1/2}_{n,\alpha}(c)}{\lambda}\right)  a_n(G_{\lambda})\varphi_n\left(\cos(\frac{\theta}{\lambda})\right)\right|^2 \sin\left(\frac{\theta}{\lambda}\right) d\theta\right]^{1/2}\\&=&\sqrt{\frac{\pi}{2}}\left[\lambda^2\sum_{n=N[\lambda]+1}^{\infty}m^2\left(\frac{\chi^{1/2}_{n,\alpha}(c)}{\lambda}\right)  a^2_n(G_{\lambda})\right]^{1/2}\\&\leq&C\,\sqrt{\frac{\pi}{2}}\left[\frac{\lambda^2}{N^4}\sum_{n=N[\lambda]+1}^{\infty}  a^2_n\left(\frac{1}{\lambda^2}\mathcal{L}(g_{\lambda})\right)\right]^{1/2}\\&\leq&\frac{C}{N^2}\,\sqrt{\frac{\pi}{2}}\left[\norm{g''+\frac{g'}{\theta}}_{L^2\left((0,\infty),\theta d\theta\right)}+C\norm{g}_{L^2\left((0,\infty),\theta d\theta\right)}\right]
\end{eqnarray*}
Then we obtain $\norm{K^N_{\lambda}}_{L^2((0,\infty),\theta d\theta)}=O(\frac{1}{N^2})$ uniformly in $\lambda.$ \\
Thus by the diagonal argument there exists a subsequence of $\{\lambda_j\}$ noted again $\{\lambda_j\}$ such that for every $N\geq 1$, $\{K^N_{\lambda_j}\}_{j\in\mathbb{N}}$ converge weakly to a function $K^N$ in $L^2((0,\infty),\theta d\theta)$ satisfy $\norm{K^N}_{L^2((0,\infty),\theta d\theta)}=O(\frac{1}{N^2})$, one conclude that there exists a sequence $\{N_k\}$ such that $\{K^{N_k}\}_{k\in\mathbb{N}}$ converge to zero almost everywhere on $(0,\infty)$.
Let $F^{N_k}=F-K^{N_k}$, clearly we have $\{F_{{\lambda_j}}^{N_k}\}_{j\in\mathbb{N}}$ converge weakly to $F^{N_k}$ in $L^2(0,\infty)$ for every $k\in \mathbb{N}$. Moreover, $\{F^{N_k}\}$ converge to $F$ almost everywhere on $(0,\infty).$\\
We prove now the following equality
\begin{equation}\label{eq11}
\lim_{j\to\infty}F_{{\lambda_j}}^{N_k}(x)=\int_0^{N_k}m(y)J_{0}(xy)\mathcal{H}_{0}(g)(y)ydy
\end{equation}
for every $x\in (0,\infty)$, the weak convergence of $\{F_{{\lambda_j}}^{N_k}\}_{j\in\mathbb{N}}$ to $F^{N_k}$, in particular, $\scal{F_{{\lambda_j}}^{N_k},\chi_{(r,s)}}$ converge to $\scal{F^{N_k},\chi_{(r,s)}}$ for every $0<r<s<\infty$ and by using the Lebesgue dominated convergence theorem which give as $\scal{F_{{\lambda_j}}^{N_k},\chi_{(r,s)}}$ converge to $\scal{\mathcal{H}_{\alpha}\left(\chi_{(0,N_k\pi)}m(.)\mathcal{H}_{\alpha}(g)\right),\chi_{(r,s)}}$, one conclude that $F^{N_k}=\mathcal{H}_{\alpha}\left(\chi_{(0,N_k\pi)}m(.)\mathcal{H}_{\alpha}(g)\right)$ almost everywhere on $(0,\infty).$ Finally, as $k\to \infty$, we get our purpose. \\
For the proof of \eqref{eq11}, we need the uniform approximation of the family of CPSWFs on $(0,1)$ which is given by the following estimates
\begin{equation}
\varphi_{n,c}^{\alpha}(\cos(\frac{\theta}{\lambda}))=(-1)^nB_n\left(\cos(\frac{\theta}{\lambda})\right)^{\alpha+1/2}P_n^{(0,\alpha)}\left(\cos(\frac{2\theta}{\lambda})\right)+\gamma_n^{\alpha+1/2}O(\frac{c^2}{n})
\end{equation}
for every $\theta\in\mathbb(\lambda t_n,\lambda\frac{\pi}{2})$, where $t_n=\arccos(\gamma_n)$ and $\gamma_n\sim\frac{\sqrt{\alpha^2-1/4}}{\chi_{n,\alpha}^{1/2}(c)}$.
\begin{equation}\label{eq12}
\varphi_{n,c}^{\alpha}(\cos(\frac{\theta}{\lambda}))=A_n\,\chi^{1/4}_{n,\alpha}(c)\frac{\sqrt{S(\cos(\frac{\theta}{\lambda}))}J_{0}\left(\chi^{1/2}_{n,\alpha}(c)S(\cos(\frac{\theta}{\lambda}))\right)}{(\sin(\frac{\theta}{\lambda}))^{\frac{1}{2}}r_n\left(\cos(\frac{\theta}{\lambda})\right)^{1/4}}+R_n(\cos(\frac{\theta}{\lambda}))
\end{equation}
for every $\theta\in\mathbb(0,\lambda t_n)$, where $A_n\sim 1,$ $r_n(t)=1-qt^2+\frac{1/4-\alpha^2}{\chi_{n,\alpha}^{1/2}(c)t^2}$ and  $\displaystyle\sup_{\theta\in\mathbb(0,t_n)}\left|R_n(\cos(\theta))\right|\leq \frac{C}{\chi^{1/2}_{n,\alpha}(c)}$, for more details see \cite{Karoui-Mehrzi}.\\
By a straightforward computation,  we have 
$$\frac{\sqrt{S(\cos(\frac{\theta}{\lambda}))}}{(\sin(\frac{\theta}{\lambda}))^{1/2}r_n\left(\cos(\frac{\theta}{\lambda})\right)^{1/4}}=1-\beta(q)(1-\cos(\frac{\theta}{\lambda}))+o(1-\cos(\frac{\theta}{\lambda}))$$
then, we can easily check that 
$$\varphi_{n,c}^{\alpha}(\cos(\frac{\theta}{\lambda}))=\chi_{n,\alpha}^{1/4}(c)J_{0}\left(\frac{\chi_{n,\alpha}^{1/2}(c)}{\lambda}\theta\right)+R_n(\cos(\frac{\theta}{\lambda})).$$
Let $N>0$ and $\lambda>\max\{\frac{2M}{\pi},N^3\}$. By \eqref{eq12}  we have, for every $n\leq N[\lambda]$
\begin{eqnarray*}
	a_n(G_{\lambda})&=&\scal{G_{\lambda},\varphi_{n,c}^{\alpha}}_{L^2(0,1)}\\&=&\frac{1}{\lambda}\int_0^{\lambda\frac{\pi}{2}}\left(g_{\lambda}\circ\arccos\right)(\cos(\frac{\theta}{\lambda}))\varphi_{n,c}^{\alpha}(\cos(\frac{\theta}{\lambda}))\sin(\frac{\theta}{\lambda})d\theta\\&=&\frac{\chi_{n,\alpha}^{1/4}(c)}{\lambda} \int_{0}^{\lambda \frac{\pi}{2}}J_{0}\left(\frac{\chi^{1/2}_{n,\alpha}(c)}{\lambda}\theta\right)g(\theta)\sin(\frac{\theta}{\lambda})d\theta-\frac{\chi_{n,\alpha}^{1/4}(c)}{\lambda} \int_{\lambda t_n}^{\lambda \frac{\pi}{2}}J_{0}\left(\frac{\chi^{1/2}_{n,\alpha}(c)}{\lambda}\theta\right)g(\theta)\sin(\frac{\theta}{\lambda})d\theta\\&+&\frac{(-1)^nB_n}{\lambda} \int_{\lambda t_n}^{\lambda \frac{\pi}{2}}g(\theta)P_n^{(0,\alpha)}\left(\cos(\frac{2\theta}{\lambda})\right)\left(\cos(\frac{\theta}{\lambda})\right)^{\alpha+1/2}\sin(\frac{\theta}{\lambda})d\theta+\frac{1}{n^{a}\chi^{1/4}_{n,\alpha}(c)}O(\frac{1}{\lambda^{b}})\\&=&\frac{\chi_{n,\alpha}^{1/4}(c)}{\lambda^2}\mathcal{H}_{0}(g)(\frac{\chi^{1/2}_{n,\alpha}(c)}{\lambda})+\frac{1}{n^{a}\chi^{1/4}_{n,\alpha}(c)}O(\frac{1}{\lambda^{b}}).
\end{eqnarray*}
where $a>1$ and $b>0.$
Indeed, using the fact that $\sup_{x>0}|\sqrt{x}J_{\alpha}(x)|\leq C_{\alpha}$, see \cite{A.YA. OLenko}, and $\lambda\sin(\frac{\theta}{\lambda})\leq \theta$ we have
\begin{eqnarray*}
	\left|\frac{\chi_{n,\alpha}^{1/4}(c)}{\lambda} \int_{\lambda t_n}^{\lambda \frac{\pi}{2}}J_{0}\left(\frac{\chi^{1/2}_{n,\alpha}(c)}{\lambda}\theta\right)g(\theta)\sin(\frac{\theta}{\lambda})d\theta\right|&\leq&\frac{1}{\lambda^{3/2}}\int_{\lambda t_n}^{\lambda \frac{\pi}{2}}\left|\frac{\chi_{n,\alpha}^{1/4}(c)}{\lambda^{1/2}}J_{0}\left(\frac{\chi^{1/2}_{n,\alpha}(c)}{\lambda}\theta\right)\right||g(\theta)|\theta d\theta\\&\leq&\frac{1}{\lambda^{3/2}}\left[\int_{\lambda t_n}^{\lambda \frac{\pi}{2}}\left|\frac{\chi_{n,\alpha}^{1/4}(c)\theta^{1/2}}{\lambda^{1/2}}J_{0}\left(\frac{\chi^{1/2}_{n,\alpha}(c)}{\lambda}\theta\right)\right|^2\frac{d\theta}{\theta}\right]^{1/2} \norm{\theta g}_{L^2(0,\infty)}\\&\leq&\frac{C_0}{\lambda^{3/2}}(\ln(\frac{\pi}{2t_n}))^{1/2} \norm{\theta g}_{L^2(0,\infty)}\\&\leq&\frac{C_0}{\lambda^{3/2}\chi_{n,\alpha}^{1/4}(c)} \norm{\theta g}_{L^2(0,\infty)}. 
\end{eqnarray*}
Moreover, using the fact that $\left|P_n^{(0,\alpha)}(\cos(\frac{2\theta}{\lambda}))\right|\leq P_n^{(0,\alpha)}(1)=O( n^{\alpha})$ and the deacreasing  cosinus function  with $|B_n|=O(n^{1/2})$, we obtain\\

$	\left|\displaystyle\frac{(-1)^nB_n}{\lambda} \int_{\lambda t_n}^{\lambda \frac{\pi}{2}}g(\theta)P_n^{(0,\alpha)}\left(\cos(\frac{2\theta}{\lambda})\right)\left(\cos(\frac{\theta}{\lambda})\right)^{\alpha+1/2}\sin(\frac{\theta}{\lambda})d\theta\right|$
\begin{eqnarray*}
	&\leq&\frac{|B_n|}{\lambda^2} \int_{\lambda t_n}^{\lambda \frac{\pi}{2}}\left|g(\theta)\right|\left|P_n^{(0,\alpha)}\left(\cos(\frac{2\theta}{\lambda})\right)\right|\left(\cos(\frac{\theta}{\lambda})\right)^{\alpha+1/2}\theta d\theta\\&\leq&\frac{C}{\lambda^2\chi^{1/4}_{n,\alpha}(c)} \norm{\theta^{3/2} g}_{L^2(0,\infty)}
\end{eqnarray*}
Finally, there exist a constant $a>1$ and $b>0$ such that
\begin{equation}
a_n(G_{\lambda})=\frac{\chi_{n,\alpha}^{1/4}(c)}{\lambda^2}\mathcal{H}_{0}(g)(\frac{\chi^{1/2}_{n,\alpha}(c)}{\lambda})+\frac{1}{n^a\chi^{1/4}_{n,\alpha}(c)}O(\frac{1}{\lambda^b}).
\end{equation}
Hence, we obtain\\
$F_{{\lambda_j}}^{N_k}(\theta)=\chi_{(0,\frac{\pi}{2}\lambda_j)}(\theta)\displaystyle\sum_{n=0}^{N_k[\lambda_j]}m\left(\frac{\chi^{1/2}_{n,\alpha}(c)}{\lambda_j}\right) a_n(G_{\lambda_j})\varphi^{\alpha}_{n,c}(\cos(\frac{\theta}{\lambda_j}))$
\begin{eqnarray*}
	&=&\sum_{n=0}^{N_k[\lambda_j]}m\left(\frac{\chi^{1/2}_{n,\alpha}(c)}{\lambda_j}\right)a_n(G_{\lambda_j})\left(\chi_{(0,\lambda_jt_n)}(\theta)\varphi_n(\cos(\frac{\theta}{\lambda_j}))+\chi_{(\lambda_j t_n,\frac{\pi}{2}\lambda_j)}(\theta)\varphi_n(\cos(\frac{\theta}{\lambda_j}))\right)\\&=&\sum_{n=0}^{N_k[\lambda_j]}m\left(\frac{\chi^{1/2}_{n,\alpha}(c)}{\lambda_j}\right)\left(\frac{\chi_{n,\alpha}^{1/4}(c)}{\lambda^2_j}\mathcal{H}_{0}(g)(\frac{\chi^{1/2}_{n,\alpha}(c)}{\lambda_j})+\frac{1}{n^{a}\chi^{1/4}_{n,\alpha}(c)}O(\frac{1}{\lambda_j^b})\right)\chi_{(0,\lambda_j\frac{\pi}{2})}(\theta)\chi_{n,\alpha}^{1/4}(c)J_{0}\left(\frac{\chi_{n,\alpha}^{1/2}(c)}{\lambda}\theta\right)\\&+&\sum_{n=0}^{N_k[\lambda_j]}m\left(\frac{\chi^{1/2}_{n,\alpha}(c)}{\lambda_j}\right)a_n(G_{\lambda_j})\chi_{(\lambda_j t_n,\frac{\pi}{2}\lambda_j)}(\theta)\left((-1)^nB_n\left(\cos(\frac{\theta}{\lambda})\right)^{\alpha+1/2}P_n^{(0,\alpha)}\left(\cos(\frac{2\theta}{\lambda})\right)-\chi_{n,\alpha}^{1/4}(c)J_{0}\left(\frac{\chi_{n,\alpha}^{1/2}(c)}{\lambda}\theta\right)\right)\\&+&O(\frac{1}{\lambda^{\varepsilon}_j})\\&=&\chi_{(0,\frac{\pi}{2}\lambda_j)}(\theta)\frac{1}{\lambda_j}\sum_{n=0}^{N_k[\lambda_j]}m\left(\frac{\chi^{1/2}_{n,\alpha}(c)}{\lambda_j}\right)\frac{\chi^{1/2}_{n,\alpha}(c)}{\lambda_j}J_{0}\left(\frac{\chi^{1/2}_{n,\alpha}(c)}{\lambda_j}\theta\right)\mathcal{H}_{0}(g)(\frac{\chi^{1/2}_{n,\alpha}(c)}{\lambda_j})+O(\frac{1}{\lambda^{\varepsilon}_j}). 
\end{eqnarray*}
where $\varepsilon>0.$ Indeed, from \cite{Szego}, we have
\begin{eqnarray*}
	(-1)^nB_n\left(\cos(\frac{\theta}{\lambda})\right)^{\alpha+1/2}P_n^{(0,\alpha)}\left(\cos(\frac{2\theta}{\lambda})\right)&=&(2n+\alpha+1)^{1/2}\left(\cos(\frac{\theta}{\lambda})\right)^{1/2}\left(\frac{\theta/\lambda}{\sin(\theta/\lambda)}\right)^{1/2}J_0\left(2(n+\frac{\alpha+1}{2})\frac{\theta}{\lambda}\right)\\&+&\frac{1}{\lambda^{1/2}}O(\frac{(2\theta)^{1/2}}{n})\\&=&(2n+\alpha+1)^{1/2}J_0\left((2n+\alpha+1)\frac{\theta}{\lambda}\right)+O(\frac{1}{n}),
\end{eqnarray*}
and by using \eqref{boundschi2}, one gets $\chi_{n,\alpha}(c)\sim (2n+\alpha+1)^2$, and concludes that
$$(-1)^nB_n\left(\cos(\frac{\theta}{\lambda})\right)^{\alpha+1/2}P_n^{(0,\alpha)}\left(\cos(\frac{2\theta}{\lambda})\right)-\chi_{n,\alpha}^{1/4}(c)J_{0}\left(\frac{\chi_{n,\alpha}^{1/2}(c)}{\lambda}\theta\right)=O(\frac{1}{n}).$$
Further, from \cite{Karoui-Boulsane}, we have $\norm{\varphi^{\alpha}_{n,c}}_{L^{\infty}(0,1)}=O(\chi^{1/2}_{n,\alpha}(c))$, then we obtain $a_n(G_{\lambda_j})=O(\frac{\chi^{1/2}_{n,\alpha}(c)}{\lambda^2}).$\\ 
Finally, as $j\to \infty$, we get $F^{N_k}=\mathcal{H}_{0}\left(\chi_{(0,N_k)}m(.)\mathcal{H}_{0}(g)\right).$	
\end{proof}

\end{document}